\documentclass[
submission
]{dmtcs-episciences}

\usepackage[utf8]{inputenc}
\usepackage{subcaption}
\usepackage{lineno,hyperref}

\newcommand*{\Union}{\bigcup} 
\usepackage{amssymb}
\usepackage{amsthm}
\usepackage{latexsym}
\usepackage{amsmath}
\usepackage{xcolor}
\usepackage{mathrsfs}
\usepackage{mathtools}
\usepackage{graphicx}
\usepackage{tabularx}
\usepackage{url}
\usepackage{tikz}
\usepackage[title]{appendix}
\usepackage[linesnumbered,ruled]{algorithm2e}
\newcommand{\col}{\text{col}}
\newcommand{\row}{\text{row}}
\usepackage{upgreek}
\usepackage{microtype}
\usepackage{stmaryrd}
\usepackage{enumerate}
\usepackage{comment}
\usepackage{times}
\usepackage{t1enc}
\usepackage{upgreek}
\usepackage{xcolor}
\usepackage{pdflscape}

\newtheorem{Corollary}{Corollary}
\newtheorem{Theorem}{Theorem}
\newtheorem{Lemma}{Lemma}

\newtheorem{Definition}{Definition}
\newtheorem{Proposition}{Proposition}
\newtheorem{Example}{Example}
\newtheorem{Remark}{Remark}

\author{Sivasankar M\affiliationmark{1}\thanks{Corresponding Author}
  \and Rama R\affiliationmark{1}}
\title[Factors of Fibonacci Sequences of $2D$ words]{Fibonacci Sequences of $1D$,$2D$ Words: Enumerating and Locating the Factors of the Fixed Points}
\affiliation{
  Department of Mathematics, Indian Institute of Technology Madras, Chennai, India}
\begin{document}
\publicationdetails{}{}{}{}{}
\maketitle

\begin{abstract}
Given an infinite word,  enumerating its factors is an important exercise for understanding the structure of the word. The process of finding all the factors is quite tricky for two-dimensional words. In this paper, two possible ways of enumerating the factors of the fixed point ($f_{\infty,\infty}$) of the sequence of Fibonacci arrays and a method for locating these factors in $f_{\infty,\infty}$ are explored. In addition, the factor complexity and the locations of the factors of the fixed point of Fibonacci sequence of arrays are also analysed.  

\end{abstract}

\section{Introduction}

Let $w$ be finite/infinite word over an alphabet $\Sigma$. The details about the subwords (otherwise called, the factors) of $w$ would be of considerable use for a better understanding of the structure and characteristics of $w$. The number of factors and the periodic (or primitive) nature of the word $w$ are closely related and are in general analysed simultaneously. Any additional information about the factors of $w$ can help in the factorization/decomposition of $w$. In turn, factorizations like Lyndon, Ziv–Lempel and Crochemore are used in text compression algorithms \cite{Burcroff:2020,Jahannia:2019}. 

Fibonacci words (more generally Sturmian words) are "simple" morphic words. By "simple", we mean that the morphisms defining these words are short and are easily conceivable. Also, it is known that, for infinite words $w$, which are not ultimately periodic, $p_{n}(w) \ge n+1$, where $p_{n}(w)$ is the number of factors of length $n$ of $w$ \cite{Berstel:1986}. It is interesting to note that Sturmian words are a class of aperiodic infinite words that achieve the least possible $p_{n}$ value, namely $n+1$ \cite{Lothaire:2002}.

Generating the Fibonacci words over $\{0,1\}$ can be systematically achieved either by the famous Fibonacci morphism, $\phi(1) = 10$, $\phi(0)=1$ \cite{Lothaire:2002} or by recursive constructions like $f_{0} = 1, f_{1} =10, f_{n} = f_{n-1}f_{n-2}, n \ge 2$ \cite{Chuan:1992,Chuan:1993,Chuan:1995}. Infinite iterations of the Fibonacci morphism or the recursion, generates the infinite Fibonacci word $f_{\infty} = 10110101\ldots$.  Some remarkable properties of $f_n$ and $f_{\infty}$ are:
(i) $f_{\infty}$  contains no fourth power,
(ii) if  a  word $u^2$ is a  factor  of $f_{\infty}$, then $u$ is  a  conjugate  of some finite Fibonacci word,
(iii) The finite Fibonacci words are primitive \cite{Berstel:1986, Luca:1981}. 

With a minimum number of subwords of any particular length, it is no wonder that the subwords occur again and again, at various locations, in $f_{\infty}$ \cite{Chuan:2005, Walczak:2010, Mignosi:1992}. There are a few interesting systematic ways to list these subwords. In \cite{Berstel:1986}, subwords of length $k$ are used to list the subwords of length $k+1$. In \cite{Rytter:2006}, a directed acyclic word graph ($DAWG$) is used to analyse the subwords. In \cite{Chuan:2005}, the suffixes of the conjugates of a specific conjugate of a finite Fibonacci word is used to find all the factors of a given length.  

As a natural extension to the one-dimensional words, two-dimensional words are studied \cite{Giam:1997,Rosenfeld:1979,Siromoney:1973}. We will interchangeably use $1D$ for one-dimensional and $2D$ for two-dimensional, hereafter in this article. Two-dimensional words finding some useful applications in image processing, data compression and crystallography is another push for exploring two-dimensional words. In \cite{Apostolico:2000}, $2D$ Fibonacci words, $f_{m,n}, m,n \ge 0$, are introduced to show that they attain the general upper bound for the number of occurrences of a particular type of tandem. In \cite{Kulkarni:2019, Mahalingam:2018}, a few combinatorial and palindromic properties of $f_{m,n}$ are studied. In \cite{Sivasankar:2023}, the authors obtain $f_{\infty,\infty}$, the $2D$ infinite Fibonacci word, using a $2D$ morphism. Further, they count the number of tandems occurring in it. 

In this paper, we list all the subwords of a given size $(k,l), k,l \ge 1$ of the $2D$ infinite Fibonacci word $f_{\infty,\infty}$. We systematically extend the methods used in \cite{Rytter:2006, Chuan:2005} for finding the subwords of $f_{\infty}$,  to $f_{\infty,\infty}$. In the later part of the paper, by using strings (of length two or more) for symbols in the Fibonacci sequence of arrays, we obtain sequences of Fibonacci arrays, and we investigate the factors of the fixed points of such sequences.

The remaining of the paper is organized as follows. In Section \ref{secpreli}, all the required definitions and notions are elaborated. In Section \ref{secdawg}, a $DAWG$ for $f_{\infty,\infty}$ is constructed and the subwords of $f_{\infty,\infty}$ are enumerated. In Section \ref{secconj}, given a $k \ge 2$ and a $l \ge 2 $, to list all the subwords of size $(k,l)$, the conjugates of a special conjugate of $f_{m,n}$ ($m,n$ depend on $k,l$) are used. In Section \ref{secloc}, the location of the factors of $f_{\infty,\infty}$ are found out. Section \ref{sec_uv_words} analyses the factors and the locations of the factors of the fixed points of the Fibonacci sequences of $1D$ words. Section \ref{sec_2d_uv_words} extends the concepts of Section \ref{sec_uv_words} to two dimensions. Finally, Section \ref{secconcl} has a few concluding remarks.

\section{Preliminaries}\label{secpreli}
\subsection{One-dimensional Words}
In formal language theory, $\Sigma$, an alphabet is a finite set of symbols and $\Sigma^*$ is the free monoid generated by $\Sigma$. The elements of $\Sigma^*$ are called words and are obtained by concatenating symbols from $\Sigma$. The neutral element of $\Sigma^*$ is the empty word (denoted by $\lambda$) and we have $\Sigma^+ = \Sigma^* - \{\lambda\}$. For a word $u \in \Sigma^*$, $|u|$ called the length of the word is the number of letters occurring in $u$. By definition, $|\lambda| = 0$. Given a word $w \in \Sigma^{*}$, $u \in  \Sigma^{*}$ is a prefix (suffix, respectively) of $w$, if $w = uv$ ($w=vu$, respectively) for some $v \in \Sigma^{*}$. The reversal of a word $u=a_{1}a_{2} \cdots a_{n}$, $a_{i} \in \Sigma$, for $1 \leq i \leq n$, is the word $u^{R}=a_{n} \cdots a_{2} a_{1}$. A word $u$ is said to be a palindrome (or a one-dimensional palindrome) if $u=u^{R}$. 

A word $w$ is said to be primitive if $w=u^{n}$ implies $n=1$ and $w=u$. Note that a power of a word is nothing but repeated concatenation of the word with itself. That is $u^n$ is obtained by concatenating $u$ with itself $n$ times. For a detailed study of formal language theory and combinatorics on words, the reader is referred to \cite{Lothaire:1997}.

\subsection{Two-dimensional Words}
The concepts of formal language theory can be obviously extended to two dimensions \cite{Giam:1997}. A two-dimensional word is called a picture or array and is a  rectangular array of symbols taken from $\Sigma$.  

\begin{Definition} \cite{Kulkarni:2019}
A $2D$ word $u=[u_{i,j}]_{1 \leq i \leq m,1 \leq j \leq n}$ of size $(m,n)$ over  $\Sigma$ is a two-dimensional rectangular finite arrangement of letters:
$$u=
\begin{matrix}
u_{1,1} & u_{1,2} & \cdots & u_{1,n-1} & u_{1,n} \\
u_{2,1} & u_{2,2} & \cdots & u_{2,n-1} & u_{2,n} \\
\vdots & \vdots & \ddots & \vdots & \vdots \\
u_{m-1,1} & u_{m-1,2} & \cdots & u_{m-1,n-1} & u_{m-1,n} \\
u_{m,1} & u_{m,2} & \cdots & u_{m,n-1} & u_{m,n} \\
\end{matrix}$$
\end{Definition}

We denote the number of rows and columns of $u$ by $|u|_{\text{row}}$ and $|u|_{\text{col}}$, respectively. An empty array, denoted by $\Lambda$ is an array of size $(0,0)$. Note that the arrays of size $(m,0)$ and $(0,m)$ for $m>0$ are not defined. The set of all arrays over $\Sigma$ including $\Lambda$, is denoted by $\Sigma^{**}$ and $\Sigma^{++}$ will denote the set of all non-empty arrays over $\Sigma$. Any subset of $\Sigma^{**}$ is called a picture language. 

To locate any position or region in an array, we require a reference system \cite{Anselmo:2014}. Given an array $u$, the set of coordinates $\{1,2,\ldots,|u|_{\text{row}}\} \times \{1,2,\ldots,|u|_{\text{col}}\}$ is referred to as the domain of $u$. A subdomain or subarray of an array $u$ (that is, a factor of the $2D$ word $u$), denoted by $u[(i,j),(i',j')]$, is the portion of $u$ located in the region $\{i,i+1,\ldots,i'\} \times \{j,j+1,\ldots,j'\}$, where $1 \leq i \leq i' \leq |u|_{\row},1 \leq j \leq j' \leq |u|_{\col}$. 

Similar to the concatenation operation in one dimension, the column concatenation and the row concatenation operations between two arrays are as follows.

\begin{Definition}
\cite{Giam:1997} Let $u,v$ be arrays over $\Sigma$ of sizes $(m_{1},n_{1})$ and $(m_{2},n_{2})$, respectively with $m_{1},n_{1},m_{2},n_{2}>0$. Then, the column concatenation of $u$ and $v$, denoted by $\obar$, is a partial operation, defined if $m_{1}=m_{2}=m$, and is given by
$$u \obar v=
\begin{matrix}
u_{1,1} & \cdots & u_{1,n_{1}} & v_{1,1} & \cdots & v_{1,n_{2}}  \\
\vdots & & \vdots & \vdots & & \vdots \\
u_{m,1} & \cdots & u_{m,n_{1}} & v_{m,1} & \cdots & v_{m,n_{2}}
\end{matrix}.$$

Similarly, the row concatenation of $u$ and $v$, denoted by $\ominus$, is another partial operation, defined if $n_{1}=n_{2}=n$, and  is given by 
$$u \ominus v=
\begin{matrix}
u_{1,1} & \cdots & u_{1,n}  \\
\vdots & & \vdots \\
u_{m_{1},1} & \cdots & u_{m_{1},n} \\
v_{1,1} & \cdots & v_{1,n} \\
\vdots & & \vdots \\
v_{m_{2},1} & \cdots & v_{m_{2},n}
\end{matrix}.$$

The column and row concatenation of $u$ and the empty array $\Lambda$ are always defined and $\Lambda$ is a neutral element for both the operations.
\end{Definition}

For a $u \in \Sigma^{**}$, an array $v \in \Sigma^{**}$ is said to be a prefix of $u$ $($suffix of $u$, respectively$)$, if $u=(v \ominus x) \obar y$  $(u=y \obar (x \ominus v)$, respectively$)$ for some $x,y \in \Sigma^{**}$. If $x \in \Sigma^{++}$, then by $(x^{k_{1} \obar})^{k_{2} \ominus}$ we mean that the array is constructed by repeating $x$, $k_{1}$ times column-wise and $x^{k_{1} \obar}$, $k_{2}$ times row-wise. An array $w \in \Sigma^{++}$ is said to be 2D \textit{primitive} if $w=(x^{k_{1} \obar})^{k_{2} \ominus}$ implies that $k_{1}k_{2}=1$ and $w=x$ \cite{Gamard:2017}. 
\subsection{Fibonacci Words}

Fibonacci words are closely related with the Fibonacci numbers. Recall the recursive definition of the Fibonacci numerical sequence: $F(0)=1$, $F(1)=2$, $F(n)=F(n-1)+F(n-2)$ for $n \geq 2$. 
Likewise, for $\Sigma=\{a,b\}$, the sequence $\{f_{n}\}_{n \geq 0}$ of Fibonacci words, is defined recursively by $f_{0}=a$, $f_{1}=ab$, $f_{n}=f_{n-1}f_{n-2}$ for $n \geq 2$. First few words of this sequence are: $f_{0}=a, f_{1}=ab, f_{2} = aba, f_{3} = abaab, f_{4} = abaababa$. Note that $|f_{n}|=F(n)$ for $n \geq 0$. The sequence of Fibonacci words can be obtained by iterating the Fibonacci morphism $\phi: \Sigma^* \rightarrow \Sigma^*$ defined by $\phi(a) = ab, \phi(b) = a$. An infinite number of iterations of $\phi$ produces the $1D$ infinite Fibonacci word $f_{\infty}$ \cite{Lothaire:2002}. That is, 
$$\lim_{n \to \infty} \phi^n(b)= f_{\infty} = abaababa\ldots.$$

\begin{Remark}
In the literature, one can observe variations in the definitions of the Fibonacci numbers and the definitions of the Fibonacci words. That is, for the convenience of simplifying the indices used in the proofs, some authors define the Fibonacci number sequence as $F(0)=1$, $F(1)=1$, $F(n)=F(n-1)+F(n-2)$ \: for\: $n \geq 2$ and the sequence of Fibonacci words as $f_0 = b, f_1 = a,
f_n = f_{n-1}f_{n-2}$ for $n \ge 2$. In such a case, the first few words of the sequence will be: $f_0 = b, f_1 = a, f_2 = ab, f_3 =
aba, f_4 = abaab$. But, in any case, the infinite Fibonacci word obtained will be the same. So, in the arguments used in this paper, we might have used the better of the two versions accordingly.    
\end{Remark}

\begin{Remark}
We also use $f_{\infty}^{s_1,s_2}$ to denote the $1D$ infinite Fibonacci word $s_1s_2s_1s_1s_2\ldots$ over the alphabet $\{s_1,s_2\}$. Similarly, $f_{n}^{s_1,s_2}$ denotes the $1D$ finite Fibonacci word $s_1s_2s_1s_1s_2\ldots s_1s_2$ or $s_1s_2s_1s_1s_2\ldots s_2s_1$ accordingly $n$ is even or odd.
    
\end{Remark}

The extension of $1D$ Fibonacci words to $2D$ Fibonacci words is presented in \cite{Apostolico:2000}.

\begin{Definition} \label{DefnApos}
\cite{Apostolico:2000} Let $\Sigma=\{a,b,c,d\}$. The sequence of Fibonacci arrays, $\{f_{m,n}\}$ where $m,n \geq 0$, is defined as:
\begin{enumerate}
\item $f_{0,0}=\beta,f_{0,1}=\gamma,f_{1,0}=\delta,f_{1,1}=\alpha$  where $\alpha, \beta,\gamma$ and $\delta$ are symbols from $\Sigma$ with some but not all, among $\alpha,\beta,\gamma$ and $\delta$ might be identical.
\item For $k \geq 0$ and $m,n \geq 1$, $$f_{k,n+1}=f_{k,n} \obar f_{k,n-1}, \hspace{0.2cm}  f_{m+1,k}=f_{m,k} \ominus f_{m-1,k}.$$
\end{enumerate} 
\end{Definition}

For convenience, let us fix $f_{0,0}=a,f_{0,1}=b,f_{1,0}=c,f_{1,1}=d$, where some but not all of $a,b,c$ and $d$ might be identical.  For example, let us derive the $2D$ Fibonacci word $f_{2,2}$.
$$f_{2,2}=f_{1,2} \ominus f_{0,2}=(f_{1,1} \obar f_{1,0}) \ominus (f_{0,1} \obar f_{0,0}).$$ It can also be obtained by column-wise expansion,
$$f_{2,2}=f_{2,1} \obar f_{2,0}= (f_{1,1} \ominus f_{0,1}) \obar (f_{1,0} \ominus f_{0,0}).$$ Using $f_{0,0}=a,f_{0,1}=b,f_{1,0}=c,f_{1,1}=d$, $f_{2,2}$ is given by $$f_{2,2}=\: \begin{matrix}
f_{1,1} & f_{1,0} \\
f_{0,1} & f_{0,0} 
\end{matrix} \: = \: \begin{matrix}
d & c \\
b & a 
\end{matrix}.$$

We state here some properties of $f_{m,n}$ which we would use later in our proofs.

\begin{Lemma}\cite{Mahalingam:2018}\label{lemprop}
Let $f_{m,n} ,  (m,n = 0,1,2,\dotsc)$ be the sequence of $2D$ Fibonacci arrays over  $\Sigma = \{ a,b,c,d\}$, with $f_{0,0} = a, f_{0,1} = b, f_{1,0} = c, f_{1,1} = d$. Also let $\Sigma_1 = \{a,b\}$, $\Sigma_2 = \{c,d\}$, $\Sigma_1' = \{a,c\}$ and $\Sigma_2' = \{b,d\}$ such that $\Sigma = \Sigma_1 \cup \Sigma_2 = \Sigma_1' \cup \Sigma_2'$. Then,  
\begin{itemize}
\item[a.] any row of $f_{m,n}$ is a $1D$ Fibonacci word over either $\Sigma_1$ or $\Sigma_2$.

\item[b.] if $\Sigma_1 \ne \Sigma_2$ then all the rows of $f_{m,n}$, over $\Sigma_1$ are identical and all the rows of $f_{m,n}$, over $\Sigma_2$ are identical.

\item[c.] any column of $f_{m,n}$ is a $1D$ Fibonacci word over either $\Sigma_1'$ or $\Sigma_2'$.

\item[d.] if $\Sigma_1' \ne \Sigma_2'$ then all the columns of $f_{m,n}$, over $\Sigma_1'$ are identical and all the columns of $f_{m,n}$, over $\Sigma_2'$ are identical. 

\item[e.] if $\Sigma_1 = \Sigma_2 (\Sigma_1'=\Sigma_2')$, then either all the rows$($columns$)$ of $f_{m,n}$ are identical or a set of rows are identical and are complementary to the set of remaining rows $($columns, respectively$)$ which are identical.
\end{itemize}
\end{Lemma}
\subsection{The \texorpdfstring{$2D$}{} Infinite Fibonacci Word, \texorpdfstring{$f_{\infty,\infty}$}{}} \label{2DFibosection}
The sequence of $2D$ finite Fibonacci words, $\{f_{m,n}\}_{m,n \ge 0}$, in a sense, has the $2D$ infinite Fibonacci word, $f_{\infty,\infty}$, as its limit. This can be perceived by extending each row, column of any $f_{m,n}$, $m,n \ge 2$, to the $1D$ infinite Fibonacci word over the alphabet of the word present in that row, column. But this outlook is informal. Formally, in \cite{Sivasankar:2023}, the authors have defined the $2D$ infinite Fibonacci word through the $2D$ morphism,
\begin{equation}
\label{FibMap}
\mu : ~~d \rightarrow \begin{matrix}
d & c\\
b & a
\end{matrix},~~c \rightarrow \begin{matrix}
d \\
b 
\end{matrix},~~b\rightarrow \begin{matrix}
d & c
\end{matrix},~~~a  \rightarrow 
d.    
\end{equation}

For a detailed study of multidimensional morphisms, \cite{Charlier:2010} can be referred. 

Observe that the morphism defined by (\ref{FibMap}) is prolongable on $d$ and an infinite number of iterations of $\mu$ on $d$ produces $f_{\infty,\infty}$ \cite{Sivasankar:2023}. That is to say, $f_{\infty,\infty}$ is the fixed point of the morphism $\mu$. That is,

$$f_{\infty,\infty} = \lim_{n \rightarrow +\infty} \mu^n(d) = \mu^{\omega}(d).$$

First few iterations of $\mu$ on $d$ are shown below.
$$d \rightarrow \begin{matrix}
d & c\\
b & a
\end{matrix} \rightarrow \begin{matrix}
d & c & d\\
b & a & b\\
d & c & d
\end{matrix} \rightarrow \begin{matrix}
d & c & d & d & c \\
b & a & b & b & a \\
d & c & d & d & c \\
d & c & d & d & c \\
b & a & b & b & a \\
\end{matrix} \rightarrow \cdots \rightarrow  \begin{matrix}
d & c & d & d & c & d & c & d &\cdots\\
b & a & b & b & a & b & a & b & \cdots\\
d & c & d & d & c & d & c & d &\cdots\\
d & c & d & d & c & d & c & d &\cdots\\
b & a & b & b & a & b & a & b &\cdots\\
d & c & d & d & c & d & c & d &\cdots\\
b & a & b & b & a & b & a & b & \cdots\\
d & c & d & d & c & d & c & d &\cdots\\
\vdots & \vdots & \vdots &\vdots & \vdots & \vdots & \vdots & \vdots & \ddots
\end{matrix}$$

As $f_{\infty, \infty}$ is the limit of $\{f_{m,n}\}_{m,n \ge 0}$, all the properties listed in Lemma \ref{lemprop} are true for $f_{\infty,\infty}$ also.

\section{Enumeration Using Subword Graphs}\label{secdawg}

In \cite{Rytter:2006} the authors have given a way to identify the subwords of $f_{\infty}$ using a directed acyclic graph. 

The directed acyclic word graph of a word $w$, $DAWG(w)$, is the smallest
finite state automaton that recognizes all the suffixes of the word \cite{Blumer:1985}. $CDAWG(w)$, a space efficient variant of $DAWG(w)$, is obtained by compacting $DAWG(w)$ \cite{Crochemore:1997}.

In \cite{Rytter:2006}, the subwords of $f_{\infty}$ are analysed through the graph $\mathcal{G}_{\infty}$, which is, in a certain sense, a $DAWG$ of $f_{\infty} = abaababb\ldots = f_{\infty}(1,2,3,\ldots)$. The $DAWG$ is constructed as below:

Let $F(0) = 1, F(1) = 2, F(n) = F(n-1) + F(n-2)$, for $n \ge 2$, be the Fibonacci sequence (Note that for $a_n$). The nodes of $\mathcal{G}_{\infty}$ are all non-negative integers. For $i > 0$, with $F(i)$ being the $i^{th}$ Fibonacci number, the labelled edges of $\mathcal{G}_{\infty}$ are 
\begin{align*}
    (i-1) \xrightarrow[]{f_{\infty}(i)} i,\quad  F(i)-2 \; \xrightarrow[]{\; s \;} \; F(i+1) -1
\end{align*} 
where $s = a$ whenever $i$ is even and $s=b$ whenever $i$ is odd (Refer Fig. \ref{dawg1d}).
\begin{figure}[ht]
    \centering
    \includegraphics[scale=0.35]{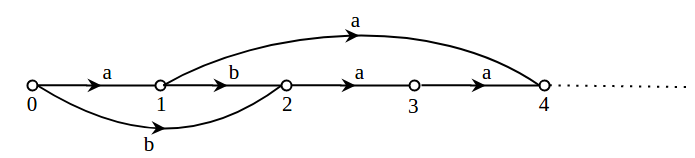}
    \caption{DAWG, $\mathcal{G}_{\infty}$ of $f_{\infty}^{a,b}$}
    \label{dawg1d}
\end{figure}
\subsection{Cross Product of \texorpdfstring{$DAWG$}{}s}
As the $2D$ finite Fibonacci words, $f_{m,n}$, can be obtained by the Cartesian product of Fibonacci reduced representation of the integers $m,n$ \cite{Kulkarni:2019}, a natural extension of $\mathcal{G}_{\infty}$ for the  $2D$ infinite Fibonacci word will be the Cartesian product of $\mathcal{G}_{\infty}$ with itself. 
\begin{Definition} \cite{West:2001}
The Cartesian product of G and H, written $G\,\square\,H$, is the graph with vertex set $V(G) \times V(H)$ specified by putting $(u, v)$ adjacent to $(u', v')$ if and only if $(1) u = u'$ and $vv' \in E(H)$, or $(2) v = v'$ and $uu' \in E(G)$. 
\end{Definition}

Since $f_{\infty, \infty}$ has two distinct rows (one over $\{d,c\}$ and one over $\{b,a\}$), to obtain a $DAWG$ of $f_{\infty, \infty}$,  we slightly modify the labels of  $\mathcal{G}_{\infty}$. Note that, all the rows of $f_{\infty, \infty}$ are $f_{\infty}$ only. In fact the rows over $\{d,c\}$ would be $dcddcdcd\ldots$ and the rows over $\{b,a\}$ would be $babbabab\ldots$. In order to simultaneously control these two categories of rows/words, we will use a single $DAWG$, the $DAWG$ of the Fibonacci word $DCDDCDCD\ldots$, with $D = \{d,b\}$ and $C = \{c,a\}$. With this adaptation, $D$ is allowed to assume either $d$ or $b$ and $C$ is allowed to assume either $c$ or $a$. As the rows of $f_{\infty, \infty}$ are words over a binary alphabet, we also impose an additional condition that, if $D$ assumes $d$ then $C$ would assume $c$ and if $D$ assumes $b$ then $C$ would assume $a$. This $DAWG$, say "\textit{$\mathcal{G}_{\infty}$ for rows}",  is depicted at the top, in Fig. \ref{crosspdt}. In the graph, for convenience, we have written $D = \{d,b\}$ and $C = \{c,a\}$ as `$d,b$'and `$c,a$', respectively.

\begin{figure}[ht]
    \centering
    \includegraphics[scale=0.35]{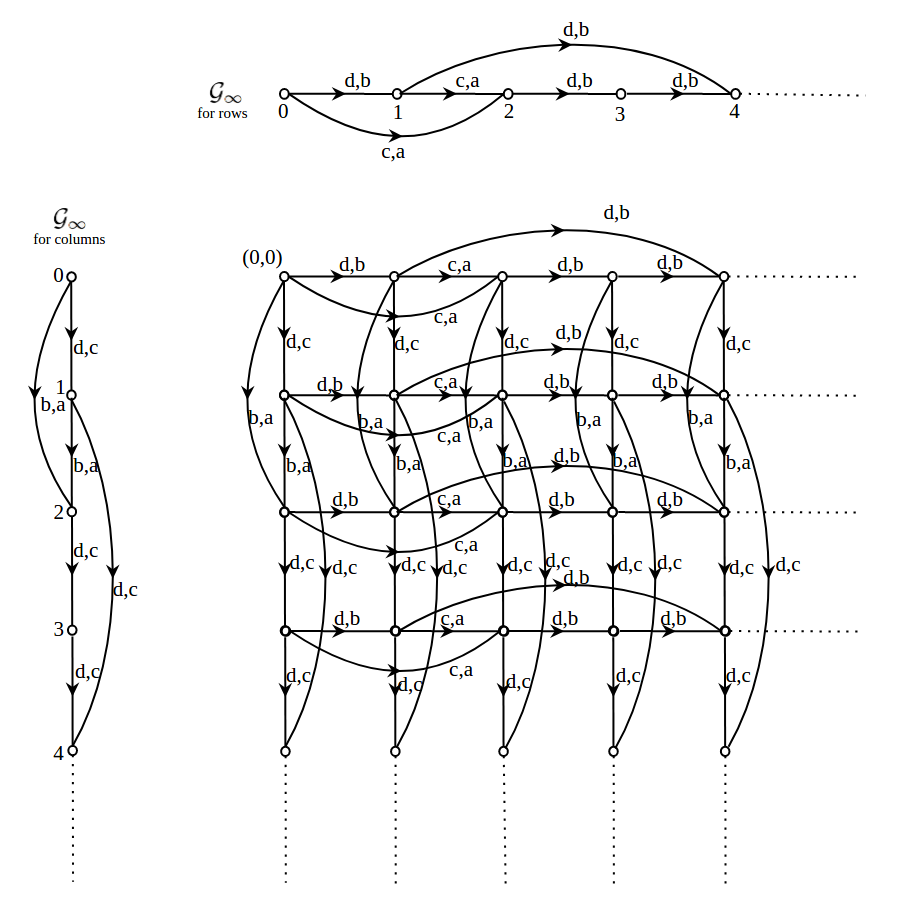}
    \caption{The Cartesian product of \textit{$\mathcal{G}_{\infty}$ for columns} and  \textit{$\mathcal{G}_{\infty}$ for rows}}
    \label{crosspdt}
\end{figure}

Similarly, since $f_{\infty, \infty}$ has two distinct columns (one over $\{d,b\}$ and one over $\{c,a\}$), to manage both the type of columns through a single $DAWG$, we consider the $DAWG$ of the Fibonacci word $D'BD'D'BD'BD'\ldots$, where $D' = \{d,c\}$ and $B = \{b,a\}$, implying $D'$ can be either $d$ or $c$, and $B$ can be either $b$ or $a$, with an additional condition that, if $D'$ is $d$ then $B$ would be $b$ and if $D'$ is $c$ then $B$ would be $a$. Again, in the graph, for convenience we write only `$d,c$' and  `$b,a$' (without the curly braces). This $DAWG$, say "\textit{$\mathcal{G}_{\infty}$ for columns}",  is depicted at the left, in Fig. \ref{crosspdt}.

Now we obtain the Cartesian product of "\textit{$\mathcal{G}_{\infty}$ for columns}" and "\textit{$\mathcal{G}_{\infty}$ for rows}". Note that, when $G$ and $H$ are labelled, the labels are carried over to the edges of the Cartesian product appropriately. The resulting graph is given in Fig. \ref{crosspdt}.

Since $1D$ words have only one direction, one can get all the letters of a subword by traversing along a directed path (starting at the root) of their $DAWG$s. But in $DAWG$s of $2D$ words, to get all the letters in a subword, all the edges that lie between the root and any node that lie in a different column/row may have to be traversed. Clearly, this is not possible as the intended $DAWG$ (that is, the Cartesian product) is acyclic and also prevents any back-and-forth traversals.

But the structure of $2D$ Fibonacci words is such that, for a subword $u$ of $f_{\infty,\infty}$, the knowledge of any one row and any one column of $u$  is enough to write down the entire $u$. Due to this, the Cartesian product will serve as the $DAWG$ of $f_{\infty,\infty}$. Further, since it is enough to know just a row and a column of $u$, even the Cartesian product is redundant and we need only the "$\textit{rooted product}$"  of "\textit{$\mathcal{G}_{\infty}$ for rows}" and "\textit{$\mathcal{G}_{\infty}$ for columns}".

\subsection{Rooted Product of \texorpdfstring{$DAWG$}{}s}

\begin{Definition}\cite{Sergey:2015}
The rooted product of a graph $G$ and a rooted graph $H$, denoted by $G \circ H$,
is defined as follows$:$ take $|V(G)|$ copies of $H$, and for every vertex $v_i$ of $G$, identify $v_i$ with the root vertex of the $i^{th}$ copy of $H$.
\end{Definition}

In other words if the vertex set of $G$ is $\{g_1,\ldots, g_n\}$ and the vertex set of $H$ is $\{h_1,\ldots,h_m\}$ with $h_ 1$ as its root, then the vertex set, $V$ and the edge set, $E$ of $G \circ H$ will be as below.
\begin{align*}
    V &= \{(g_i,h_j): 1 \le i \le n , 1 \le j \le m\}\\
    E &= E_1 \cup E_2 \quad \text{where,}\\
    E_1 &=\{((g_i,h_1),(g_k,h_1)):(g_i,g_k) \in E(G)\},\\
    E_2 &=\Union_{i=1}^n \{((g_i,h_j),(g_i,h_k)):(h_j,h_k) \in E(H)\}
\end{align*}
In fact, it is easy to see that, $G \circ H$ is a subgraph of $G \square H$. 

Now, we take the "$\textit{rooted product}$" of "\textit{$\mathcal{G}_{\infty}$ for rows}" and "\textit{$\mathcal{G}_{\infty}$ for columns}" (Refer Fig. \ref{dawg2d}) to get the $DAWG$ of $f_{\infty,\infty}$ and denote it by $\mathcal{G}_{\infty,\infty}$. From $\mathcal{G}_{\infty,\infty}$, we can obtain the first row  and the last column of any subword of $f_{\infty,\infty}$. We designate the node $(0,0)$ as the root node of $\mathcal{G}_{\infty,\infty}$.

\begin{figure}[ht]
    \centering
    \includegraphics[scale=0.35]{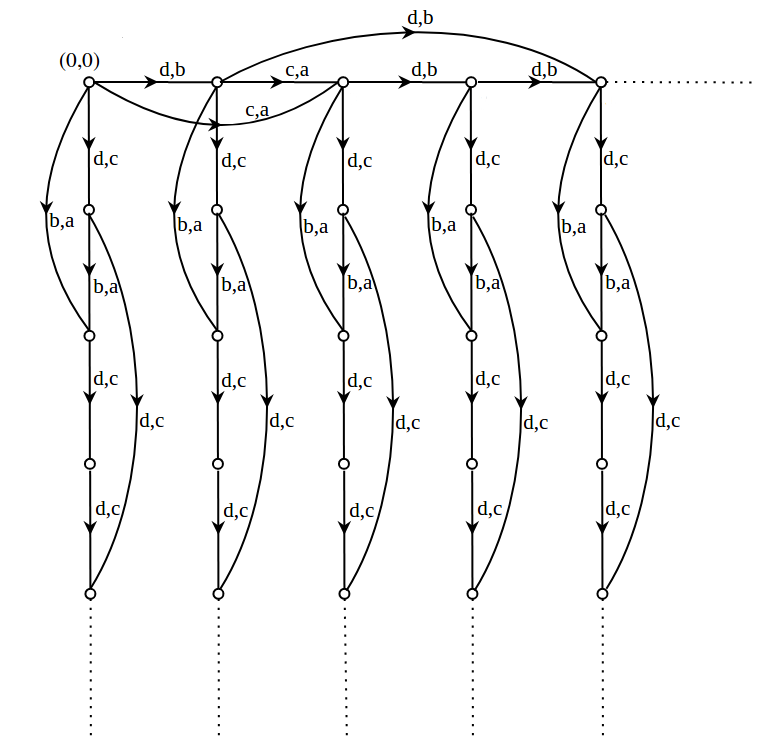}
    \caption{DAWG, $\mathcal{G}_{\infty,\infty}$ of $f_{\infty,\infty}$: (\textit{$\mathcal{G}_{\infty}$ for rows}) $\circ$ (\textit{$\mathcal{G}_{\infty}$ for columns})}
    \label{dawg2d}
\end{figure}
\subsection{Enumerating the subwords: The \texorpdfstring{$DAWG$}{} way}
In this subsection, we prove that the number of finite paths in $\mathcal{G}_{\infty,\infty}$, starting at its root node, equals the number of subwords of $f_{\infty,\infty}$. In particular, we prove that, for $k,l \ge 1$, a path of length $k+l$ , comprising of a horizontal path of length $k$ and a vertical path of length $l$, will lead to subword of $f_{\infty,\infty}$ of size $(k,l)$.  Note that by a horizontal path (a vertical path, respectively), we mean a path whose adjacent vertices are in $\textit{$\mathcal{G}_{\infty}$ for rows}$ ($\textit{$\mathcal{G}_{\infty}$ for columns}$, respectively).

\begin{Theorem} \label{kbylsubword}
Let $k,l \in \mathbb{N}$ be given. Then, from a path of length $k+l$ (starting at the root) in $\mathcal{G}_{\infty,\infty}$, comprising of a horizontal path of length $l$ and a vertical path of length $k$, we can construct a subword of $f_{\infty,\infty}$ of size $(k,l)$. 
\end{Theorem}
\begin{proof}
Due to the construction of $\mathcal{G}_{\infty,\infty}$, when we start at the root and traverse a horizontal path of length $l \ge 1$, we get a subword of the $1D$ Fibonacci infinite word $DCDDCDCD\ldots$. In fact, we can obtain two horizontal subwords of length $l$, one over $\{d,c\}$ (obtained by taking $d$ for $D$ and $c$ for $C$) and one over $\{b,a\}$ (obtained by taking $b$ for $D$ and $a$ for $C$). The former subword occurs in any row of $f_{\infty,\infty}$ which is over $\{d,c\}$, and the later occurs in any row of $f_{\infty,\infty}$ which is over $\{b,a\}$. 

Now, starting from the last node of this horizontal path, we traverse a vertical path of length $k$. Note that, the rooted product guarantees such a path. Similar to the earlier argument, here we obtain a vertical path of length $k \ge 1$, which corresponds to a subword of length $k$ of the $1D$ Fibonacci infinite word $D'BD'D'BD'BD'\ldots$. Here also we can obtain two vertical subwords of length $k$, one over $\{d,b\}$ (obtained by taking $d$ for $D'$ and $b$ for $B$) and one over $\{c,a\}$ (obtained by taking $c$ for $D'$ and $a$ for $B$). The former subword occurs in any column of $f_{\infty,\infty}$ which is over $\{d,b\}$, and the later occurs in any column of $f_{\infty,\infty}$ which is over $\{c,a\}$.

To prove that these two paths can produce a unique subword of size $(k,l)$ of $f_{\infty,\infty}$, we use the fact that \text{`}the last letter in the first row and the first letter in the last column of a $2D$ word are the same'.  Hence, while constructing the subword, the last letter (say "\textit{$s_{joint}$}") in the horizontal path has to be the first letter in the vertical path. For example, out of the two available subwords of length $l$, suppose we select the subword over $\{d,c\}$, say $H$, and if \textit{$s_{joint}$} $=d$ (\textit{$s_{joint}$} $=c$, respectively), then we will(have to) select the vertical subword , say $V$, over $\{d,b\}$ ($\{c,a\}$, respectively). Now, by taking $H$ and $V$ as the first row and the last column, respectively, in a $2D$ word of size $(k,l)$, we will obtain the entire subword. Again note that, this is not possible for all $2D$ words, but for $f_{\infty,\infty}$, due to its structure.

As any row of $f_{\infty,\infty}$ is either over $\Sigma_1 = \{a,b\}$ or $\Sigma_2 = \{c,d\}$, \textit{$s_{joint}$} has to be either in $\Sigma_1$ or in $\Sigma_2$. As any column of $f_{\infty,\infty}$ is either over $\Sigma_1' = \{a,c\}$ or $\Sigma_2' = \{b,d\}$,  $V$ has to be either in $\Sigma_1'$ or in $\Sigma_2'$.  Hence the following four cases only arise. 
\begin{center}
\begin{tabular}{ l l l }
Case (i)&:& \textit{$s_{joint}$} $=a$ (then, $V$ will be over $\{a,c\}$)\\
Case (ii)&:& \textit{$s_{joint}$} $=b$ (then, $V$ will be over $\{b,d\}$) \\
Case (iii)&:& \textit{$s_{joint}$} $=c$ (then, $V$ will be over $\{a,c\}$)\\
Case (iv)&:& \textit{$s_{joint}$} $=d$ (then, $V$ will be over $\{b,d\}$) 
\end{tabular}
\end{center}
To find the letters occurring at the other positions of $u$ we define two substitutions. If $H$ is over $\{a,b\}$, we create a $1D$ word $H'$ from $H$ using the substitution $\theta_1: \theta_1(a) = c , \theta_1(b) = d$. If $H$ is over $\{c,d\}$, we create a $1D$ word $H''$ from $H$ using the substitution $\theta_2: \theta_2(c) = a , \theta_2(d) = b$.  These words $H'$ and $H''$ will be used to fill up/find the other rows of the subword we are constructing. These substitutions are motivated by the fact that, a row of $f_{\infty,\infty}$ over $\{a,b\}$ can be obtained from a row of $f_{\infty,\infty}$ over $\{c,d\}$ and vice-versa through simple substitutions. 

Let $R_1, R_2, R_3, \ldots, R_k$ be the $k$ rows of the subword being constructed. Note that $R_1 = H$. Now, for $2 \le j \le k$,\\

\textbf{Case(i):} \textit{$s_{joint}$} $=a$ (and hence $H$ is over $\{a,b\}$)

If the letter in the $j^{th}$ row of $V$ is  \textit{$s_{joint}$}, then $R_j = H$ else $R_j = H'$.

\textbf{Case(ii):} \textit{$s_{joint}$} $=b$ (and hence $H$ is over $\{a,b\}$)

If the letter in the $j^{th}$ row of $V$ is  \textit{$s_{joint}$}, then $R_j = H$ else $R_j = H'$.

\textbf{Case(iii):} \textit{$s_{joint}$} $=c$ (and hence $H$ is over $\{c,d\}$)

If the letter in the $j^{th}$ row of $V$ is  \textit{$s_{joint}$}, then $R_j = H$ else $R_j = H''$.

\textbf{Case(iv):} \textit{$s_{joint}$}$=d$ (and hence $H$ is over $\{c,d\}$)

If the letter in the $j^{th}$ row of $V$ is  \textit{$s_{joint}$}, then $R_j = H$ else $R_j = H''$.\\

Note that while constructing the subword, the alphabet of each row and the order in which the two distinct rows ($H$ and $H'$ (or) $H$ and $H''$) of the subword are getting arranged are decided/guided by $V$. Since $V$ is a subword of length $l$ of some column of $f_{\infty,\infty}$, the obtained $2D$ word is a subword of $f_{\infty,\infty}$ of size $(k,l)$. 
\end{proof}

\begin{Remark}
Theorem \ref{kbylsubword} can be proved by taking "$\textit{rooted product}$" of "\textit{$\mathcal{G}_{\infty}$ for columns}" and "\textit{$\mathcal{G}_{\infty}$ for rows}". In that case, first we have to traverse a vertical path of length $k$, then a horizontal path of length $l$ to obtain the first column and the last row of the subword in that order. Finding the other rows can be done similar to the process explained in the proof. 
\end{Remark}

\begin{Remark}\label{rem2}
Since we constructed \textit{$\mathcal{G}_{\infty, \infty}$} as the "$\textit{rooted product}$" of "\textit{$\mathcal{G}_{\infty}$ for rows}" by "\textit{$\mathcal{G}_{\infty}$ for columns}", we will always use a horizontal edge (an edge of \textit{$\mathcal{G}_{\infty}$ for rows}) at first. Also, as $l \ge 1$, we will never use the copy of \textit{$\mathcal{G}_{\infty}$ for columns} rooted at $(0,0)$. Hence we can remove this redundant copy from \textit{$\mathcal{G}_{\infty, \infty}$} and can still entitle the new graph \textit{$\mathcal{G}_{\infty, \infty}$}.
\end{Remark}

\begin{Remark}
The $DAWG$ also can be constructed by a similar methodology as given in \cite{Rytter:2006}. Let 
$$f_{row, \infty} =DCDDCDCD\ldots = f_{row, \infty}(1,2,3,\ldots), $$ $$f_{col, \infty} =D'BD'D'BD'BD'\ldots = f_{col,\infty}(1,2,3,\ldots)$$
where $D,C,D'$ and $B$ are as defined earlier.
The nodes of $\mathcal{G}_{\infty,\infty}$ are all non-negative integer pairs, $(i,j), i,j \ge 0$. 

For $j > 0$, with $F(j)$ being the $j^{th}$ Fibonacci number, the labelled edges of $\mathcal{G}_{\infty,\infty}$ are 
\begin{align*}
    (0,j-1) \xrightarrow[]{f_{row,\infty}(j)} (0,j),\quad (0, F(j)-2) \; \xrightarrow[]{\; s \;} \; (0,F(j+1) -1)
\end{align*} 
where $s = D$ whenever $j$ is even and $s=C$ whenever $j$ is odd, and

for each $j \ge 0$ ($j\ge 1$ is suffice; refer Remark \ref{rem2}) and $i > 0$,
\begin{align*}
    (i-1,j) \xrightarrow[]{f_{col,\infty}(i)} (i,j),\quad (F(i)-2,j) \; \xrightarrow[]{\; s \;} \; (F(i+1) -1,j)
\end{align*} 
where $s = D'$ whenever $i$ is even and $s=B$ whenever $i$ is odd.
\end{Remark}

\begin{Corollary}\label{kplus1by}
For $k,l \ge 1$, there are $(k+1)(l+1)$ subwords of size $(k,l)$ in $f_{\infty,\infty}$.
\end{Corollary}

\begin{proof}
 As the graph "\textit{$\mathcal{G}_{\infty}$ for rows}" is the $DAWG$ of the $1D$ Fibonacci word $DCDDC\ldots$, there are $(l+1)$ horizontal paths in $\mathcal{G}_{\infty, \infty}$ \cite{Rytter:2006}. Since the graph "\textit{$\mathcal{G}_{\infty}$ for columns}" is the $DAWG$ of the $1D$ Fibonacci word $D'BD'D'B\ldots$, from the last node of every horizontal path of $\mathcal{G}_{\infty, \infty}$, there are $(k+1)$ vertical paths available for traversing. Note that, though paths with labels from $\{D,C\}$ and $\{D',B\}$ have two possibilities,  due to the condition on \textit{$s_{joint}$} (as explained in the proof of Theorem \ref{kbylsubword}), only one path with labels $\{a,b,c,d\}$ will materialize. Thus, there are $(k+1)(l+1)$ paths of length $k+l$, comprising of a horizontal path of length $l$ and a vertical path of length $k$. Now, by Theorem \ref{kbylsubword}, a path of length $k+l$ in $\mathcal{G}_{\infty,\infty}$, comprising of a horizontal path of length $l$ and a vertical path of length $k$, uniquely corresponds to a subword of size $(k,l)$ of $f_{\infty,\infty}$. Hence the corollary.
\end{proof}

The following example will explain the construction used in the proof of Theorem \ref{kbylsubword}.

\begin{Example}
Let $k =2$ and $l=2$ so that all the subwords of size $(2,2)$ will be obtained. By corollary \ref{kplus1by}, there will be $9$ subwords of this size. Construction of one of these 9 subwords is explained here.

The horizontal paths of length $2$ in $DCDDC\ldots$ are $DC, DD$ and $CD$. Suppose we select $d$ for $D$. Then $H$ can be any one of $\{dc,dd,cd\}$. Let us choose $H$ as $cd$.

Now the vertical paths of length $2$ in $D'BD'D'B\ldots$ are $\begin{matrix}
D'\\
B
\end{matrix}, \; \begin{matrix}
D'\\
D'
\end{matrix}$ and $\begin{matrix}
B\\
D'
\end{matrix}$.
Since \textit{$s_{joint}$} = $d$, to have a subword of $f_{\infty,\infty}$, the vertical path of length 2 should start with $d$. By selecting $d$ for $D'$ we have the three vertical paths $\left\{ \begin{matrix}
d\\
b
\end{matrix}, \begin{matrix}
d\\
d
\end{matrix},\begin{matrix}
b\\
d
\end{matrix} \right\}$.

Let us take $V=\begin{matrix}
d\\
b
\end{matrix}$. Then the first column and the last row of the subword are fixed. 
The incomplete subword is, $\begin{matrix}
c & d\\
 * & b
\end{matrix}$, where the symbol '$*$' denotes the entry therein is unknown yet.

Since $s_{joint} = d$ and the letter in the second row of $V$ is not a $d$, we fill the second row with $H'' = ab$. Hence the subword corresponding to this path is $\begin{matrix}
c & d\\
 a & b
\end{matrix}$.

All the possible 9 cases of $H$, $V$ and their corresponding subwords are listed in Tab. \ref{9subwords} .
\begin{table}[ht]
    \centering
    \caption{All the factors of size $(2,2)$ of $f_{\infty, \infty}$}
    \label{9subwords}
  
 $\begin{tabular}{|c|c|c|c|c|c|c|c|c|c|}
\hline
\textbf{H} & \quad $d c$ \quad & \quad $d c$ \quad  & \quad $d d$ \quad & \quad $d d$ \quad & \quad $c d$ \quad  & \quad $c d$ \quad  & \quad $b a$ \quad  & \quad $b b$ \quad & \quad $a b$ \quad    \\
\hline
\textbf{V} & $\begin{matrix}
c\\
a
\end{matrix}$  & $\begin{matrix}
c\\
c
\end{matrix}$ &  $\begin{matrix}
d\\
b
\end{matrix}$ & $\begin{matrix}
d\\
d
\end{matrix}$ &$\begin{matrix}
d\\
b
\end{matrix}$ & $\begin{matrix}
d\\
d
\end{matrix}$ & $\begin{matrix}
a\\
c
\end{matrix}$ & $\begin{matrix}
b\\
d
\end{matrix}$ & $\begin{matrix}
b\\
d
\end{matrix}$ \\ 
 \hline
\textbf{Incomplete} & $\begin{matrix}
d & c\\
* & a
\end{matrix}$  & $\begin{matrix}
d & c\\
* & c
\end{matrix}$  &  $\begin{matrix}
d & d\\
* & b
\end{matrix}$  & $\begin{matrix}
d & d\\
* & d
\end{matrix}$  & $\begin{matrix}
c & d\\
* & b
\end{matrix}$ & $\begin{matrix}
c & d\\
* & d
\end{matrix}$  & $\begin{matrix}
b & a\\
* & c
\end{matrix}$  & $\begin{matrix}
b & b\\
* & d
\end{matrix}$  & $\begin{matrix}
a & b\\
* & d
\end{matrix}$  \\ 
\textbf{Subword}& &  & & & & & & &  \\
 \hline
\textbf{Complete} & $\begin{matrix}
d & c\\
b & a
\end{matrix}$  & $\begin{matrix}
d & c\\
d & c
\end{matrix}$  &  $\begin{matrix}
d & d\\
b & b
\end{matrix}$  & $\begin{matrix}
d & d\\
d & d
\end{matrix}$  & $\begin{matrix}
c & d\\
a & b
\end{matrix}$ & $\begin{matrix}
c & d\\
c & d
\end{matrix}$  & $\begin{matrix}
b & a\\
d & c
\end{matrix}$  & $\begin{matrix}
b & b\\
d & d
\end{matrix}$  & $\begin{matrix}
a & b\\
c & d
\end{matrix}$  \\ 
\textbf{Subword}& &  & & & & & & &  \\
 \hline
  \end{tabular}$
\end{table}
\end{Example}

\section{Enumeration by Conjugation}\label{secconj}
 For a given $k$, let $n$ be the smallest integer such that  $1 \le k < F(n)$, where $F(n)$ is the $n^{th}$ Fibonacci number. In this section we use the method described in \cite{Chuan:2005}, wherein it is proved that the prefixes of length $k$ of the conjugates of a "special" conjugate of $f_n$ are the subwords  of length $k$ of $f_{\infty}$. The Lemma is recalled here.

With $\Sigma$, an alphabet, define the operator $T$ on $\Sigma^+$ as follows. For a word $w =a_1a_2\ldots a_n \in \Sigma^+$, $T(a_1a_2\ldots a_{n-1}a_n) = a_2\ldots a_{n-1}a_na_1$ and $T^{-1}(a_1a_2\ldots a_{n-1}a_n) = a_na_1a_2\ldots a_{n-1}$. Higher powers of $T$ are defined iteratively. That is, $T^p(w) = T(T^{p-1}(w))$ and $T^{-p}(w) = T^{-1}(T^{-(p-1)}(w))$.
\begin{Lemma}\cite{Chuan:2005}\label{1dTs}
Let $f_0 = a, f_1 = b, f_n = f_{n-1}f_{n-2}, n \ge 2$ be the sequence of Fibonacci words. Let $F(n) = |f_n|$ and let 

\[ q_n = \begin{cases} 
     T^{F(n) - 1}(f_n) & \text{if n is even} \\
     T^{F(n-1) - 1}(f_n) & \text{if n is odd}.
   \end{cases}
\]
Then for each $k$ with $1 \le k < F(n)$, the $k+1$ prefixes of $T^{0}(q_n), T^{-1}(q_n), \ldots, T^{-k}(q_n)$ having length $k$ are the $k + 1$ distinct factors of $f_{\infty}$ of length $k$.
\end{Lemma}

\begin{Example}
Let $f_{\infty} = abaab\ldots$. For $k=4$, $n$ will be $4$, as $4 < F(4)$. So, $f_n = f_4 = abaab$. With $F(4) = 5, F(3) = 3$, we have $q_4 = T^4(abaab) = babaa$, is the special conjugate of $f_4$. 

Now, $T^{0}(q_4), T^{-1}(q_4), T^{-2}(q_4), T^{-3}(q_4), T^{-4}(q_4)$ are $babaa$, $ababa$, $aabab$, $baaba$, $abaab$ respectively and the subwords of $f_{\infty}$ of length 4 are $baba,abab, aaba, baab, abaa$.
\end{Example}

Similar to the operators $T$ and $T^{-1}$, we define four operators on $2D$ words. 
\begin{Definition}
Let $r_1 , r_2, \cdots, r_m$ and $c_1 , c_2, \cdots, c_n$ be the $m$ rows and the $n$ columns of a $2D$ word $w$ of size $(m,n)$. Then the operations $T_{col}(w)$, $T^{-1}_{col}(w)$, $T_{row}(w)$ and $T^{-1}_{row}(w)$ are defined as below.
\begin{align*}
T_{col}(w) & = c_{2} \obar c_{3} \obar   \cdots \obar c_{n} \obar c_{1}\\
T^{-1}_{col}(w) & = c_{n} \obar c_{1} \obar c_{2} \obar   \cdots \obar c_{n-2} \obar c_{n-1}\\
T_{row}(w) & = r_{2} \ominus r_{3} \ominus   \cdots \ominus r_{m} \ominus r_{1} \\
T^{-1}_{row}(w) & = r_{n} \ominus  r_{1} \ominus  r_{2} \ominus    \cdots \ominus  r_{n-2} \ominus  r_{n-1}.
\end{align*}
\end{Definition} 
Higher powers of $T_{col}(w)$, $T^{-1}_{col}(w)$, $T_{row}(w)$ and $T^{-1}_{row}(w)$ are defined iteratively. For example, with $s \ge 1$, $T^s_{col}(w) = T_{col}(T^{s-1}_{col}(w))$.

Through these operators we define the conjucacy class of a $2D$ word $w$.

\begin{Definition} \label{conj}
Let $w$ be a $2D$ word of size $(m,n)$. Then 
\begin{align*}
 Conj(w) &= \left\{ T_{row}^{i} T_{col}^{j} (w) ,  0 \le i \le m-1 , 0 \le j \le n-1 \right\}  \\
   &=  \left\{ T_{col}^{j} T_{row}^{i} (w) ,  0 \le j \le n-1 , 0 \le i \le m-1 \right\}
\end{align*}
is called the Conjugacy Class of $w$.  
\end{Definition}
Since $0 \le i \le m-1$ and $0 \le j \le n-1$, it is easy to see that the number of conjugates of $w$ can be at the maximum $mn$. Note that, if no two rows of $w$ are conjugates of each other and if no two columns of $w$ are conjugates of each other, then the maximum possible value of $mn$ will be achieved by $|Conj(w)|$.

Now, we will enumerate the subwords of size $(k,l)$ of $f_{\infty,\infty}$ using the conjugates of a "special" conjugate of $f_{m,n}$ ($m,n \ge 3$ and depend on $k,l$). 

\begin{Theorem}
Let $F(0)=F(1)=1, F(2)=2, F(3)=3, F(4)=5, \ldots $ be the sequence of Fibonacci numbers. For a given $k,l \ge 1$, consider the $2D$ finite Fibonacci word $f_{m,n}$, where $m$, $n$ are the smallest integers such that $k < F(m)$ and $l < F(n)$. Let 

\[ q_{m,n} = \begin{cases} 
     T_{row}^{F(m) - 1} \left( T_{col}^{F(n) - 1}(f_{m,n})\right) & \text{if m is even and n is even} \\
     T_{row}^{F(m) - 1}\left(T_{col}^{F(n-1) - 1}(f_{m,n})\right) & \text{if m is even and n is odd}\\
      T_{row}^{F(m-1) - 1}\left(T_{col}^{F(n) - 1}(f_{m,n})\right) & \text{if m is odd and n is even} \\
       T_{row}^{F(m-1) - 1}\left(T_{col}^{F(n-1) - 1}(f_{m,n})\right) & \text{if m is odd and n is odd} 
   \end{cases}
\]
Then for each $k$ with $1 \le k < F(m)$ and for each $l$ with $1 \le l < F(n)$, the $(k+1)(l+1)$ prefixes of 
\begin{center}
 \begin{tabular}{c}
$T^{0}_{row}T^{0}_{col}(q_{m,n}), T^{0}_{row}T^{-1}_{col}(q_{m,n}), \ldots \: \dots, \: T^{0}_{row}T^{-l}_{col}(q_{m,n})$,\\
\\
$T^{-1}_{row}T^{0}_{col}(q_{m,n}), T^{-1}_{row}T^{-1}_{col}(q_{m,n}),  \ldots \: \dots, \: T^{-1}_{row}T^{-l}_{col}(q_{m,n})$,\\
$\cdots \quad \cdots \quad \cdots$\\
$\cdots \quad \cdots \quad \cdots$\\
$T^{-k}_{row}T^{0}_{col}(q_{m,n}), T^{-k}_{row}T^{-1}_{col}(q_{m,n}),  \ldots \: \ldots, \: T^{-k}_{row}T^{-l}_{col}(q_{m,n})$
\end{tabular}
\end{center}
having size $(k,l)$ are the $(k+1)(l+1)$ distinct factors of $f_{\infty,\infty}$ of size $(k,l)$.
\end{Theorem}

\begin{proof}
Suppose that we want to find all the subwords of $f_{\infty,\infty}$ of size $(k,l)$. Let $F(0)=F(1)=1, F(2)=2, F(3)=3, F(4)=5, \ldots $ be the sequence of Fibonacci numbers. Consider the $2D$ finite Fibonacci word $f_{m,n}$ where $m$ and $n$ are such that $k < F(m)$ and $l < F(n)$. Note that $f_{m,n}$ will be of size ($F(m), F(n)$) \cite{Mahalingam:2018}. 

We prove the theorem for the case where both $m$ and $n$ are even. The proofs of other cases are similar.

Denote the columns of $f_{m,n}$ by $C_1, C_2, \ldots, C_{F(n)}$. Since there are only two distinct columns (refer Lemma \ref{lemprop}), let us symbolize the columns over $\{b,d\}$ by $D$ and the columns over $\{a,c\}$ by $C$. As every row of $f_{m,n}$ is a Fibonacci word of size $F(n)$, the two distinct columns are indeed arranged in a Fibonacci pattern in $f_{m,n}$. That is, the symbolized word for $f_{m,n} = C_1 \obar C_2 \obar \ldots \obar C_{F(n)} = DCDDC\ldots DC = H_n$, say, is a Fibonacci word of size $F(n)$. Since $n$ is even, the suffix of length 2 of $H_n$ will be $DC$. Now by Lemma \ref{1dTs}, the prefixes of length $l$ of  the conjugates of $T^{F(n) - 1}(H_n) = q_n'$ (say), are the subwords of length  $l$ of $H_n$. We now replace the symbols $D$ and $C$ occurring in $q_n'$ by the original columns to get the $2D$ word $q_n$. What we have proved is that, we can arrange the columns of $f_{m,n}$ in a way that we can obtain all the subwords of length $l$ of the infinite Fibonacci words occupying the rows of $f_{m,n}$ through the conjugates of $q_n$.

Now, let us denote the rows of $q_n$ by $R_1, R_2, \ldots, R_{F(m)}$. By symbolizing the rows over $\{d,c\}$ as $D'$ and $\{a,b\}$ as $B$, we get $V_m$, the symbolized word of $q_n$ over $\{D',B\}$ as $V_m = R_1 \ominus R_2 \ominus \ldots \ominus R_{F(m)} = D'BD'D'B\ldots D'B$. Following a similar argument as above, we get a word $q_{m}' = T^{F(m) - 1}(V_m)$ over $\{D',B\}$. We can now replace the symbols occurring in $q_m'$ to get the $2D$ word $q_{m,n}$. What we have proved is that, we can arrange the rows of $q_{n}$ in a way that we can get all the subwords of length $k$ of the infinite Fibonacci words occupying the columns of $f_{m,n}$ through the conjugates of $q_{m,n}$.

Note that $q_{m,n}$ is a conjugate of $f_{m,n}$. In fact, by the two stage process, what we have obtained as $q_{m,n}$ is nothing but $T_{row}^{F(m) - 1} \left( T_{col}^{F(n) - 1}(f_{m,n})\right)$. As assured by Lemma \ref{1dTs}, the rows and columns of $q_{m,n}$ are arranged in such a way that, for each $0 \le i \le k$, $0 \le j \le l$, the prefixes of length $k$ of the first columns and the prefixes of length $l$ of the first rows of  $T^{i}_{row}T^{j}_{col}(q_{m,n})$,  produces $(k+1)(l+1)$ distinct  $FRAME_{TL}$s. We can call $q_{m,n}$ a "special" conjugate of $f_{m,n}$, in this context. Since in each of these $FRAME_{TL}$s, $FRAME_{L}$s are subwords of $f_{\infty}^{d,b}$ or $f_{\infty}^{c,a}$,  and $FRAME_{T}$s are subwords of $f_{\infty}^{d,c}$ or $f_{\infty}^{b,a}$, by Lemma \ref{frmframtl} we get $(k+1)(l+1)$ distinct subwords of $f_{\infty,\infty}$.
\end{proof}

Let us understand the enumeration of the subwords through an example.
\begin{Example}
Let $k= 2$, $l=2$. That is, we wish to find all the $(2,2)$-subwords of $f_{\infty,\infty}$. As $k$ and $l$ are less than $F(3)= 3$, $m=n=3$ and we consider 
\begin{align*}
f_{3,3} = \begin{matrix}
d & c & d\\
b & a & b\\
d & c & d
\end{matrix}.   
\end{align*}
Then, as both $m$ and $n$  are odd,  $q_{3,3} = T_{row}^{2 - 1} \left( T_{col}^{2 - 1}(f_{3,3})\right) = f_{3,3} = \begin{matrix}
a & b & b\\
c & d & d\\
c & d & d
\end{matrix}. $

As mentioned earlier $q_{3,3}$ is a "special" conjugate of $f_{3,3}$. Since no two rows(columns) of $q_{3,3}$ are conjugates of each other, $q_{3,3}$ has $9$ distinct conjugates. All the $9$ conjugates and their corresponding subwords of size $(2,2)$ of  $f_{\infty,\infty}$ are listed in Table \ref{subbyconj}.

\begin{table}[htp]
    \centering
    \caption{Conjugates of $q_{3,3}$ and the subwords of size $(2,2)$ of $f_{\infty, \infty}$}
    \label{subbyconj}
  
 $\begin{tabular}{|c|c|c|}
\hline
&&\\

 \boldmath{\textbf{$T_{row}^{i}(T_{col}^{j}(q_{3,3}))$}} & \textbf{Conjugate of $q_{3,3}$}  & \textbf{Subword} \\ 
&&\\

 \hline
$T_{row}^{0}(T_{col}^{0}(q_{3,3}))$& $\begin{matrix}
a & b & b\\
c & d & d\\
c & d & d
\end{matrix}$  & $\begin{matrix}
a & b\\
c & d
\end{matrix}$    \\ 

 \hline
$T_{row}^{0}(T_{col}^{-1}(q_{3,3}))$ & $\begin{matrix}
b & a & b  \\
d & c & d \\
d & c & d 
\end{matrix}$   & $\begin{matrix}
b & a\\
d & c
\end{matrix}$  \\
 \hline
 
$T_{row}^{0}(T_{col}^{-2}(q_{3,3}))$ & $\begin{matrix}
 b & b & a  \\
 d & d & c \\
 d & d & c 
\end{matrix}$   & $\begin{matrix}
b & b\\
d & d
\end{matrix}$  \\
 \hline
 
$T_{row}^{-1}(T_{col}^{0}(q_{3,3}))$ & $\begin{matrix}
c & d & d\\
a & b & b \\
c & d & d
\end{matrix}$  & $\begin{matrix}
c & d\\
a & b
\end{matrix}$    \\ 

 \hline
$T_{row}^{-1}(T_{col}^{-1}(q_{3,3}))$ & $\begin{matrix}
d & c & d \\
b & a & b  \\
d & c & d 
\end{matrix}$   & $\begin{matrix}
d & c\\
b & a
\end{matrix}$  \\
 \hline
 
 $T_{row}^{-1}(T_{col}^{-2}(q_{3,3}))$ & $\begin{matrix}
  d & d & c \\
  b & b & a  \\
  d & d & c 

\end{matrix}$   & $\begin{matrix}
d & d\\
b & b
\end{matrix}$  \\
 \hline
 $T_{row}^{-2}(T_{col}^{0}(q_{3,3}))$ & $\begin{matrix}
c & d & d \\
c & d & d \\
a & b & b

\end{matrix}$  & $\begin{matrix}
c & d\\
c & d
\end{matrix}$    \\ 

 \hline
$T_{row}^{-2}(T_{col}^{-1}(q_{3,3}))$ & $\begin{matrix}
 d & c & d  \\
  d & c & d  \\
  b & a & b 

\end{matrix}$   & $\begin{matrix}
d & c\\
d & c
\end{matrix}$  \\
 \hline
 
 $T_{row}^{-2}(T_{col}^{-2}(q_{3,3}))$ & $\begin{matrix}
  d & d & c \\
  d & d & c \\
  b & b & a 
\end{matrix}$   & $\begin{matrix}
d & d\\
d & d
\end{matrix}$  \\
 \hline
 
  \end{tabular}$
     
\end{table}

\end{Example}

In \cite{Chuan:2005}, apart from the sophisticated way of obtaining the subwords of $f_{\infty}$, described in Lemma \ref{1dTs}, the author provides another simple way of obtaining the subwords of length $k$. 

\begin{Proposition}\cite{Chuan:2005}\label{sbwrdaspre}
Let $n \ge 2$ and $F(n) \le k < F(n+1)$. Then, the prefixes of length $k$ of $T^{i}(f_{n+1})$,  $i \in \{0,1,\ldots,F(n)-1\} \cup$ $\{F(n+2)-k-1,F(n+2)-k, \ldots, F(n+1)-1\}$,  are the $k+1$ distinct factors of $f_{\infty}$ of length $k$.
\end{Proposition}
Proposition \ref{sbwrdaspre} is extended  to $f_{\infty,\infty}$ as below. 

\begin{Proposition}
Let $m,n \ge 2$ and $F(m) \le k < F(m+1)$, $F(n) \le l < F(n+1)$. Then the $(k+1)(l+1)$ prefixes of $T_{row}^{i} (T_{col}^{j}(f_{m+1,n+1}))$ of size $(k,l)$, where $i \in \{0,1,\ldots,F(m)-1\} \cup \{F(m+2)-k-1,F(m+2)-k, \ldots, F(m+1)-1\}$, $j \in \{0,1,\ldots,F(n)-1\} \cup \{F(n+2)-l-1,F(n+2)-l, \ldots, F(n+1)-1\}$, are the $(k+1)(l+1)$ distinct factors of $f_{\infty,\infty}$ of size $(k,l)$.
\end{Proposition}
\begin{proof}
For $m,n \ge 2$, consider the $2D$ finite Fibonacci word $f_{m+1,n+1}$. Recall that the columns and rows of $f_{m+1,n+1}$ are $1D$ finite Fibonacci words (in fact, they are $f_{m+1}^{d,b}$ or $f_{m+1}^{c,a}$,  and $f_{n+1}^{d,c}$ or $f_{n+1}^{b,a}$). Hence, the $FRAME_{L}$s and $FRAME_{T}$s of the $(k+1)(l+1)$ $FRAME_{TL}$s obtained from the conjugates $T_{row}^{i} (T_{col}^{j}(f_{m+1,n+1}))$, $i \in \{0,1,\ldots,F(m)-1\} \cup \{F(m+2)-k-1,F(m+2)-k, \ldots, F(m+1)-1\}$, $j \in \{0,1,\ldots,F(n)-1\} \cup \{F(n+2)-l-1,F(n+2)-l, \ldots, F(n+1)-1\}$ are nothing but the $(k+1)(l+1)$ appropriate combinations of vertical factors of length $k$ and horizontal factors of length $l$ of the infinite Fibonacci words occurring in the columns and in the rows of $f_{\infty,\infty}$. Since all these $FRAME_{TL}$s are taken from the prefixes of $f_{\infty,\infty}$, the subword constructed from these $FRAME_{TL}$s (refer Lemma \ref{frmframtl}) will be obviously prefixes of $f_{\infty,\infty}$. Hence, the prefixes of size $(k,l)$ of $T_{row}^{i} (T_{col}^{j}(f_{m+1,n+1}))$ for the stated values of $i,j$ are the factors of $f_{\infty,\infty}$ of size $(k,l)$.
\end{proof}

\section{Locating the Factors of \texorpdfstring{$f_{\infty,\infty}$}{}}\label{secloc}
In the previous sections, we developed two methods for listing all the $(k+1)(l+1)$ factors of size $(k,l)$ of $f_{\infty,\infty}$. In this section we will locate (find the exact positions $\{(i,j), i,j \ge 1\}$ of) these factors in the domain of $f_{\infty,\infty}$. We know that since there are only $k+1$ factors of length $k$ in $f_{\infty}$, there are many repetitions of every factor in $f_{\infty}$ \cite{Berstel:1986}. As the rows and columns of $f_{\infty,\infty}$ are composed of $f_{\infty}$, the same happens in $f_{\infty,\infty}$ also.   

For locating the factors of $f_{\infty}$, the reader may either refer \cite{Chuan:2005} or \cite{Rytter:2006}. We recall some terminologies from \cite{Rytter:2006} for our use. 

Let $f_{0} = a, f_1 = ab$ and for $n \ge 1,  f_{n+1} = f_{n}f_{n-1}$ so that $f_{\infty} = abaababaabaab\ldots \ldots = f_{\infty}(1,2,3,\ldots)$. Also, for $n \ge 0$ let $F(n)= |f_{n}|$ be the $n^{th}$ Fibonacci number. For $n \ge 2$, let $g_n$ be the $n^{th}$ truncated Fibonacci word, the word obtained from $f_n$ by removing its last two letters.

Let $u$ be a subword of $f_{\infty}$. By an \textit{occurrence} of $u$ we mean a $i \ge 0$ such that $f_{\infty}(i+1)f_{\infty}(i+2) \dots f_{\infty}(i+|u|) = u$. By \textit{first-occ($u$)} we mean the least value of \textit{occurrence} of $u$ and by \textit{occ($u$)} we mean the set of all \textit{occurrences} of $u$ in $f_{\infty}$. Now, for a set of integers $X$ and for a $j \ge 0$, define the operator $\boxplus$ as, $X \boxplus j = \{x + j: x \in X\}$. 

Recall that the Fibonacci number system represents a number as a sum of Fibonacci numbers such that no two consecutive Fibonacci numbers are used. Also, the sum of zero number of integers equals zero. This representation of any nonnegative integer $n$, in the Fibonacci number system is called the Fibonacci representation of $n$. For $n \ge 1$, let $\mathcal{Z}_n$ be the set of nonnegative integers which do not use Fibonacci numbers $F(0), F(1), F(2), \ldots, F(n-1)$ in their Fibonacci representation. For example $\mathcal{Z}_1 = \{0,2,3,5,\ldots\}$ and $\mathcal{Z}_2 = \{0,3,5,8,11,\ldots\}$. Then, it is proved in \cite{Rytter:2006} that,

\begin{equation} \label{1stocc}
 \textit{occ}(u) = occ(g_n) \boxplus \textit{first-occ}(u),   
\end{equation}

where $n$ is such that $g_n$ is the shortest truncated Fibonacci word containing $u$. Since for $n \ge 2$, \textit{occ}($g_{n+1}$) = \textit{occ}($f_{n}$) = $\mathcal{Z}_{n}$, we have,
$$\textit{occ}(u) = \mathcal{Z}_{n-1} \boxplus \textit{first-occ}(u).$$
We also have that \textit{occ}($f_{1}$) = \textit{occ}($f_{2}$) and \textit{occ}($f_{0}$)$ = \mathcal{Z}_{1}$.
\begin{Example} \label{exam5}
Let us locate the positions of the factor $u = abab$ in $f_{\infty}^{a,b}$. We have \textit{first-occ($u$)} = $3$. Since $u$ occurs for the first time in $g_5$, we get $n=5$ and hence \textit{occ}($abab$) = $\mathcal{Z}_{4} \boxplus 3  = \{0,8,13,21,29,\ldots\} \boxplus 3 = \{3,11,16,24,32,\ldots\}$.
\end{Example}

For locating the factors of $f_{\infty,\infty}$, let us define a structure called "\textit{FRAME}". 
\begin{Definition} \label{frame}
Let $w$ be a $2D$ word. The structure obtained by considering only the first row, the first column, the last row and the last column of $w$ is called the \textit{FRAME} of $w$. In particular, the first row (first column, last row, and last column, respectively) is called \textit{$FRAME_T$} (\textit{$FRAME_L$}, \textit{$FRAME_B$}, and \textit{$FRAME_R$}, respectively).     
\end{Definition}

It is understood that by $FRAME_T, FRAME_L, \ldots$, we refer to the words they contain. Extending Definition \ref{frame}, we can have the substructures  \textit{$FRAME_{TL}$} (which consists of \textit{$FRAME_T$} and \textit{$FRAME_L$}), \textit{$FRAME_{TR}$}, \textit{$FRAME_{LB}$} and \textit{$FRAME_{RB}$}. We will be predominantly using \textit{$FRAME_{TL}$} only.  Note that \textit{$FRAME_T$} and \textit{$FRAME_L$} share a common prefix of length one. We call this common symbol \textit{$s_{joint,TL}$}. Similarly \textit{$s_{joint,TR}$} is defined (Refer Fig. \ref{framesfig}). 

\begin{figure}[ht]
\centering
     \begin{subfigure}{0.49\linewidth}
         \centering
         \includegraphics[width=0.8\linewidth]{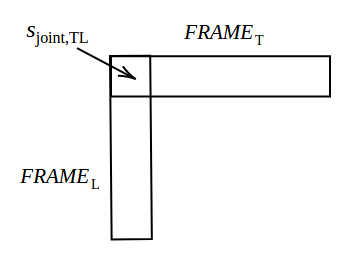}
         \subcaption{$FRAME_{TL}$}
      \end{subfigure}
     \begin{subfigure}{0.49\linewidth}
         \centering
         \includegraphics[width=0.8\linewidth]{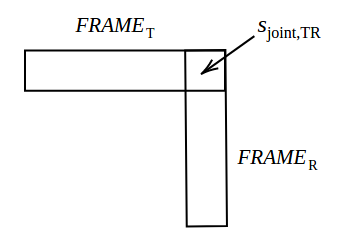}
         \subcaption{$FRAME_{TR}$}
     \end{subfigure}
     \caption{Two of the four substructures of $FRAME$}
     \label{framesfig}
\end{figure}

The following Lemma is inspired by the properties listed in Lemma \ref{lemprop}. Note that there are only two distinct rows in $f_{\infty,\infty}$. These distinct rows also are one and the same words except that their respective alphabets are different. Hence, given the entire first row and any one letter of another row $R$, row $R$ can be written down with ease using a substitution rule.

\begin{Lemma} \label{frmframtl}
Given \textit{$FRAME_{TL}$} of a subword of $f_{\infty,\infty}$ with its $FRAME_{T}$ being a subword of length $l$ of $f_{\infty}^{a,b}$ or $f_{\infty}^{c,d}$  and $FRAME_{L}$ being a subword of length $k$ of $f_{\infty}^{a,c}$ or $f_{\infty}^{b,d}$, we can construct the subword of size $(k,l)$ of $f_{\infty,\infty}$ with that $FRAME_{TL}$.
\end{Lemma}
\begin{proof}
 Let $u$ be the $2D$ word whose \textit{$FRAME_{TL}$} is given. We will make use of the two substitution rules defined in the proof of Theorem \ref{kbylsubword} to get the factor $u$ of $f_{\infty,\infty}$. 

If \textit{$FRAME_{T}$} is over $\{a,b\}$, then define $\theta_1: \theta_1(a) = c , \theta_1(b) = d$. Now, for any row $i$, $2 \le i \le k$, of \textit{$FRAME_{L}$}, if the letter present therein is \textit{$s_{joint,TL}$}, the $i^{th}$ row of $u$ is \textit{$FRAME_{T}$} itself; else, the $i^{th}$ row of $u$ is $\theta_1(\textit{$FRAME_{T}$})$. 

If \textit{$FRAME_{T}$} is over $\{c,d\}$, then define $\theta_2: \theta_2(c) = a , \theta_2(d) = b$. Now, for any row $i$, $2 \le i \le k$, of \textit{$FRAME_{L}$}, if the letter present therein is \textit{$s_{joint,TL}$}, the $i^{th}$ row of $u$ is \textit{$FRAME_{T}$} itself; else, the $i^{th}$ row of $u$ is $\theta_2(\textit{$FRAME_{T}$})$. 

As mentioned in the proof of Theorem \ref{kbylsubword}, the rows other than \textit{$FRAME_{T}$} are constructed using \textit{$FRAME_{L}$}. That is the alphabet of a particular row and the order in which the two distinct rows of the subword are arranged are decided by \textit{$FRAME_{L}$}. Since \textit{$FRAME_{L}$} is a subword of length $l$ of some column of $f_{\infty,\infty}$, the obtained $2D$ word is a subword of $f_{\infty,\infty}$ of size $(k,l)$. 
\end{proof}

\begin{Remark}
In Lemma \ref{frmframtl}, we have constructed the entire subword from \textit{$FRAME_{TL}$}. Similarly, with appropriate conditions on \textit{$FRAME_{L}$}, \textit{$FRAME_{T}$}, \textit{$FRAME_{R}$}, \textit{$FRAME_{B}$}, one can construct the entire subword from any of \textit{$FRAME_{LB}$},\textit{$FRAME_{TR}$} and \textit{$FRAME_{RB}$} also.
\end{Remark}

We are now ready to locate any factor of $f_{\infty,\infty}$. Let $w$ be a subword of $f_{\infty,\infty}$. Let the size of $w$ be $(k,l)$. Note that, because $w$ is a $2D$ word, \textit{first-occ($w$)} will be a pair $(i,j)$ such that \textit{first-occ}($FRAME_T$ of $w$) is $j$ in the $i^{th}$ row of $f_{\infty,\infty}$ and \textit{first-occ}($FRAME_L$ of $w$) is $i$ in the $j^{th}$ column of $f_{\infty,\infty}$ and thus the domain of $w$ in $f_{\infty,\infty}$ is $\{i+1,i+2,\ldots,i+k\} \times \{j+1,j+2,\dots,j+l\}$. The definition of \textit{occ($w$)} is similar to its $1D$ counterpart.  Since a subword of $f_{\infty,\infty}$ is uniquely determined by its $FRAME_{TL}$, its first occurrence and hence its all other occurrences will be determined by the  first occurrences of its $FRAME_{T}$ and $FRAME_{L}$. With $FRAME_{T}$ and $FRAME_{L}$ both being subwords of $1D$ Fibonacci words, we have the following Proposition.

\begin{Proposition}\label{firstocc}
Let $w$ be a subword of $f_{\infty,\infty}$. Let $FRAME_{T}$ and $FRAME_{L}$ denote its first row and first column respectively. Then, 
\begin{align*}
    \textit{first-occ($w$)}  = 
    \begin{cases}
    (fo_{L}^{d,b}, fo_{T}^{d,c}) &\quad\text{if $FRAME_{L}$ is over $\{d,b\}$ and $FRAME_{T}$ is over $\{d,c\}$}\\
    (fo_{L}^{d,b}, fo_{T}^{b,a}) &\quad\text{if $FRAME_{L}$ is over $\{d,b\}$ and $FRAME_{T}$ is over $\{b,a\}$}\\
    (fo_{L}^{c,a}, fo_{T}^{d,c}) &\quad\text{if $FRAME_{L}$ is over $\{c,a\}$ and $FRAME_{T}$ is over $\{d,c\}$}\\
    (fo_{L}^{c,a}, fo_{T}^{b,a}) &\quad\text{if $FRAME_{L}$ is over $\{c,a\}$ and $FRAME_{T}$ is over $\{b,a\}$}\\
     \end{cases}
\end{align*}  
where $fo_{L}^{d,b}$ is the \textit{first-occ}($FRAME_{L}$) in $f_{\infty}^{d,b}$, $fo_{L}^{c,a}$ is the \textit{first-occ}($FRAME_{L}$) in $f_{\infty}^{c,a}$, $fo_{T}^{d,c}$ is the \textit{first-occ}($FRAME_{T}$) in $f_{\infty}^{d,c}$ and $fo_{T}^{b,a}$ is the \textit{first-occ}($FRAME_{T}$) in $f_{\infty}^{b,a}$.
\end{Proposition}
\begin{proof}
We discuss the proof for the case in which $FRAME_{L}$ is over $\{d,b\}$ and $FRAME_{T}$ is over $\{d,c\}$. Proofs of the other cases are similar. 

Since $w$ can occur in $f_{\infty,\infty}$, only when $FRAME_{TL}$ of $w$ occurs in $f_{\infty,\infty}$, it is clear that \textit{first-occ($w$)} is decided by  \textit{first-occ}($FRAME_{L}$) in $f_{\infty}^{d,b}$ and \textit{first-occ}($FRAME_{T}$) in $f_{\infty}^{d,c}$. Let $fo_{L}^{d,b} \ge 0$, be \textit{first-occ}($FRAME_{L}$) in $f_{\infty}^{d,b}$. Let $fo_{T}^{d,c} \ge 0$ be \textit{first-occ}($FRAME_{T}$) in $f_{\infty}^{d,c}$. Since all the columns of $f_{\infty,\infty}$ over $\{d,b\}$ are identical $fo_{L}^{d,b}$ value will be the same in all the columns which are over $\{d,b\}$. So in the ${fo_{T}^{d,c}}^{\:th}$ column (where $FRAME_{T}(w)$ occurs for the first time) also, $fo_{L}^{d,b}$ will be the same. Similarly, since all the rows of $f_{\infty,\infty}$ over $\{d,c\}$ are identical $fo_{T}^{d,c}$ value will be the same in all the rows which are over $\{d,c\}$. 

Now, \textit{first-occ}($FRAME_{TL}(w)$) can be $(i,j)$, say, only when both \textit{first-occ}($FRAME_{L}(w)$) and \textit{first-occ}($FRAME_{T}(w)$) are $(i,j)$. Hence \textit{first-occ($w$)} = \textit{first-occ}($FRAME_{TL}(w)$) = ($fo_{L}^{d,b}$, $fo_{T}^{d,c}$),\: if $FRAME_{L}$ is over $\{d,b\}$ and $FRAME_{T}$ is over $\{d,c\}$.
\end{proof}
\begin{Corollary}
Let $w$ be a subword of $f_{\infty,\infty}$. Let \textit{first-occ($w$)} be given by Proposition \ref{firstocc}. Let $FRAME_{L}(w)$ be over $\{s_1,s_2\}$ and $FRAME_{T}(w)$ be over $\{s'_1,s'_2\}$. Then, \textit{occ}($w$) = $X \times Y$, where $X = \mathcal{Z}_{m-1} \boxplus fo_{L}^{s_1,s_2}$ and $Y=\mathcal{Z}_{n-1} \boxplus fo_{T}^{s'_1,s'_2}$.
 
\end{Corollary}
\begin{proof}
In all the columns of $f_{\infty,\infty}$ which are  over $\{s_1,s_2\}$, \textit{occ}($FRAME_{L}$) = $\mathcal{Z}_{m-1} \boxplus fo_{L}^{s_1,s_2}$, where $m$ is such that $g_m$ is the shortest truncated Fibonacci word over $\{s_1,s_2\}$ containing $FRAME_{L}$. Similarly, in all the rows of $f_{\infty,\infty}$ which are  over $\{s_1',s_2'\}$, \textit{occ}($FRAME_{T}$) = $\mathcal{Z}_{n-1} \boxplus fo_{T}^{s_1',s_2'}$, where $n$ is such that $g_n$ is the shortest truncated Fibonacci word over $\{s_1',s_2'\}$ containing $FRAME_{T}$. Therefore, \textit{occ}($FRAME_{TL}(w)$) = (\textit{occ}($FRAME_{L}(w)$),\:\textit{occ}($FRAME_{T}(w)$)) = $\{(x,y): x \in \mathcal{Z}_{m-1} \boxplus fo_{L}^{s_1,s_2}, y \in \mathcal{Z}_{n-1} \boxplus fo_{T}^{s_1',s_2'}\} = X \times Y$ where $X = \mathcal{Z}_{m-1} \boxplus fo_{L}^{s_1,s_2}$ and $Y=\mathcal{Z}_{n-1} \boxplus fo_{T}^{s'_1,s'_2}$. Since \textit{occ}($w$) =  \textit{occ}($FRAME_{TL}(w)$), the result follows.
\end{proof}
\begin{Example}
Let us find the \textit{occ($w$)} where $w = \begin{matrix}
d&d&c\\
d&d&c\\
b&b&a
\end{matrix}$. 

Note that, $FRAME_{L}$ of $w$ is \: $\begin{matrix}
d\\
d\\
b
\end{matrix}$ \: and \textit{first-occ}($FRAME_{L}$ ) in $f_{\infty}^{d,b}$ is $2$. That is $fo_{L}^{d,b} = 2$. Also the value of $m$ such that $g_m^{d,b}$ contains $FRAME_{L}$ is $4$. Similarly, $FRAME_{T}$ of $w$ is \: "$\begin{matrix}
d& d& c \end{matrix}$" \: and \textit{first-occ}($FRAME_{T}$) in $f_{\infty}^{d,c}$ is $fo_{T}^{d,c} = 2$. The value of $n$ such that $g_n^{d,c}$ contains $FRAME_{T}$ is $4$. 

Therefore, \textit{first-occ}($w$) = $(2,2)$. And, \textit{occ($w$)} = $X \times Y$, where where $X = \mathcal{Z}_{3} \boxplus 2$ and $Y=\mathcal{Z}_{3} \boxplus 2$. With $\mathcal{Z}_{3} = \{0,5,8,13,18,\ldots\}$, we have $X = \{2,7,10,15,20,\ldots\}$ and $Y = \{2,7,10,15,20,\ldots\}$. Hence \textit{occ($w$)} = $\{(2,2), (2,7), (2,10), \ldots, (7,2), (7,7), (7,10),  \ldots\}$.   
\end{Example}

\section{Fibonacci sequence of \texorpdfstring{$1D$}{} words}\label{sec_uv_words}
In this section we discuss the factor complexity of the Fibonacci language $F_{u,v}$ where $u,v \in \{a,b\}^+$ and $|u|,|v| \ge 2$. We find the bounds of the factor complexity function and the location of the factors of the fixed point of the Fibonacci sequence of words.  

\begin{Definition}\cite{Shyr:2005}
    Let $\Sigma$ be a finite alphabet consisting of more than one element. For two words $u,v \in \Sigma^+$ the following two types of Fibonacci sequences of words can be defined.
    \[ (1)\text{~} w_0 = u, w_1 = v, w_2 = vu, \ldots, w_n=w_{n-1}w_{n-2} \ldots;\]
    \[ (2)\text{~} w_0' = u, w_1' = v, w_2' = uu, \ldots, w_n'=w_{n-2}'w_{n-1}' \ldots;\]
\end{Definition}

For example, with $\Sigma = \{a,b\}, u=abba$ and $v=bba$ we have,  $w_0 = abba, w_1 = bba, w_2 = bbaabba, w_3= bbaabbabba, \ldots$. The languages $F_{u,v} = \{w_i \mid i \ge 0\}$ and $F'_{u,v} = \{w_i' \mid i \ge 0\}$ are called Fibonacci Languages. As $F_{u,v}$ and $F'_{u,v}$ are similar, it is enough to study $F_{u,v}$. In \cite{Shyr:2005}, primitive and palindromic words in $F_{u,v}$ are studied.

Note that if we denote by $\Gamma$ the alphabet $\{u,v\}$ then $w_0,w_1,w_2,\ldots$ are the $1D$ Fibonacci words over $\Gamma$ generated by the familiar Fibonacci morphism $g: u \rightarrow v, \quad v \rightarrow vu$. Denoting the fixed point of this sequence of Fibonacci words by $f_{\infty,u,v}$ we have, $f_{\infty;u,v} = vuvvuvuv\cdots$. Now, for $k,k' \ge 2$, suppose we have $u = u_1u_2\cdots u_k,v = v_1v_2\cdots v_{k'}$, with $u_1,\cdots, u_k,v_1,\cdots, v_{k'} \in \Sigma$, then we get the fixed point of this Fibonacci sequence of words as an infinite word over $\{a,b\}$. Let us denote this fixed point by $f_{\infty;a,b}$. 

\subsection{Factor Complexity of \texorpdfstring{$f_{\infty;a,b}$}{}} \label{FacCompFibSeq}

Here, as a first step, we study the factor complexity of $f_{\infty;a,b}$ under the condition that $|u|=|v|$ (i.e. when $k = k'$).

\begin{Theorem}\label{T3}
    Let $k \ge 2$ and $|u|=|v|=k$. Let $m$ denote the length of the factors of $f_{\infty;u,v}$ and let $l$ denote the length of the factors of $f_{\infty;a,b}$. Given an $l \ge 2$, consider the least $m \ge 2$ such that $km \ge l$. Then for $(m-1)k+2 \le l \le mk+1$, we have,

\[   
p_{f_{\infty;a,b}}(l)  \le 
     \begin{cases}
       \text{$(k-1)(m+1)+(1)(m+2)$,} &\quad\text{if $l=(m-1)k+2$}\\
       \text{$(k-2)(m+1)+(2)(m+2)$,} &\quad\text{if $l=(m-1)k+3$}\\
       \hspace{1.5cm}\text{$\vdots$} &\hspace{1.5cm}\text{$\vdots$}\\
       \text{$(1)(m+1)+(k-1)(m+2)$,} &\quad\text{if $l=mk$}\\
       \text{$(k)(m+2)$,} &\quad\text{if $l=mk+1$}\\
     \end{cases}
\]
That is, for $l = (m-1)k + i + 1$, $1 \le i \le k$, we have $p_{f_{\infty;a,b}}(l)  \le (k-i)(m+1) + (i)(m+2)$. 
\end{Theorem}
\begin{proof}
    Let $m \ge 2$ denote the length of the factors of $f_{\infty;u,v}$ and let $l \ge 2$ denote the length of the factors of $f_{\infty;a,b}$. We analyze $p_{f_{\infty;a,b}}(l)$ iteratively as $m$ increases from $2$, in steps of $1$. At every iterative stage we count the number of new factors created by appending either $u = u_1u_2\cdots u_k$  or $v = v_1v_2\cdots v_{k}$ and update $l$. Observe that only these two symbols can be appended to the existing factor (of length $m-1$) of $f_{\infty;u,v}$. In other words we analyze the factors  of  $f_{\infty;a,b}$ through the factors of $f_{\infty;u,v}$.
    
    Let us visualize $f_{\infty;u,v}$ and $f_{\infty;a,b}$ as shown below. \\

$f_{\infty;u,v} = \begin{tabular}{|c|c|c|c|c|}
\hline
$v$&$u$&$v$&$v$&$u$\\
\hline
\end{tabular} \cdots \cdots.$\\

 $f_{\infty;a,b} = \begin{tabular}{|c|c|c|c|c|}
\hline
$v_1v_2\cdots v_k$&$u_1u_2\cdots u_k$&$v_1v_2\cdots v_k$&$v_1v_2\cdots v_k$&$u_1u_2\cdots u_k$\\
\hline
\end{tabular} \cdots \cdots.$\\

Recall that $f_{\infty;u,v}$ being the infinite Fibonacci word has $m+1$ factors of length $m$. For an easy understanding, let us elaborate the counting process for $m=2$. Consider any one of the three factors $vu,uv,vv$. Let us take $vu = \begin{tabular}{|c|c|}
\hline
$v_1v_2\cdots v_k$&$u_1u_2\cdots u_k$\\
\hline
\end{tabular}$. The following table can be constructed easily by observing the starting and the ending positions of the new factors created while appending $u$ with $v$.

\begin{table}[h!]
\centering
\begin{tabular}{|c|c|c|}
\hline
\textbf{Factors of $f_{\infty;a,b}$}&  \textbf{Length of the factor ($l$)} &\textbf{Number of factors} \\
\hline
$v_1v_2\cdots v_ku_1$,~
$v_2v_3\cdots v_ku_1u_2$,~$\cdots$&  $k+1$ &  $k$ \\
\hline
$v_2v_3\cdots v_ku_1$,~
$v_3v_4\cdots v_ku_1u_2$,~$\cdots$&  $k$ &  $k-1$ \\
\hline
$\cdots$&  $\cdots$ &  $\cdots$ \\
\hline
$v_{k-1}v_ku_1$, $v_ku_1u_2$&  $3$ &  $2$ \\
\hline
$v_ku_1$&  $2$ &  $1$ \\
\hline
\end{tabular}
\caption{Factors of $f_{\infty;a,b}$ formed when $m=2$}
\label{mis2factors}
\end{table}

This counting has to be done for each of the three factors possible ($vu,uv,vv$)  and hence the values in the `Number of factors' column in Tab. \ref{mis2factors} are to be multiplied by $3$. Now, a few more factors of the same lengths, listed above, will be created by the factors of length $3$ of $f_{\infty; u,v}$ also. For example, from $vuv = \begin{tabular}{|c|c|c|}
\hline
$v_1v_2\cdots v_k$&$u_1u_2\cdots u_k$&$v_1v_2\cdots v_k$\\
\hline
\end{tabular}$\:, we get one factor of length $(k+2)$ (namely, $v_ku_1u_2\cdots u_kv_1$) and two factors of length $(k+3)$ and so on. Note that, for a given $l$ and an appropriate $m$, the factors of length $l$ are (inherently) available  at the beginning of a factor of length $m$ of $f_{\infty;u,v}$ and are available at the middle of a factor of length $(m+1)$ of $f_{\infty;u,v}$.

Extending this counting technique, for an $m \ge 2$, we have,
\begin{center}
\begin{tabular}{|c|c|c|}
\hline
\textbf{Length ($l$) of the} & \textbf{No. of factors created } & \textbf{No. of factors created }   \\
 \textbf{factor of $f_{\infty;a,b}$}& \textbf{by a factor of length $m$} & \textbf{by a factor of length $m+1$} \\
  & \textbf{of $f_{\infty;u,v}$} & \textbf{of $f_{\infty;u,v}$} \\
\hline
$(m-1)k+2$ &  $k-1$ &  $1$ \\
\hline
$(m-1)k+3$ &  $k-2$ &  $2$  \\
\hline
$\cdots$ &  $\cdots$  &  $\cdots$   \\
\hline
$mk$ &  $1$ &  $k-1$ \\
\hline
$mk+1$ &  $0$ &  $k$\\
\hline
\end{tabular}
\end{center}

Now for an $m \ge 2$, as there are $m+1$ factors of length $m$ and $m+2$ factors of length $m+1$ in $f_{\infty;u,v}$, we get the bound for $p_{f_{\infty;a,b}}(l)$ as stated in the theorem. That is, by adding together the number of factors of length $(m-1)k+2 \le l \le mk+1$ that occur in the factors of length $m$ and $m+1$ of $f_{\infty;u,v}$ we get, for $m \ge 2$,   

\begin{table}[h!]
\centering
\begin{tabular}{|c|c|}
\hline
\textbf{Length of the factor ($l$)} & \textbf{Maximum number of factors}  \\
\hline
$(m-1)k+2$ &  $ (k-1)(m+1) + 1(m+2)$ \\
\hline
$(m-1)k+3$ &   $ (k-2) (m+1) + 2(m+2)$ \\
\hline
$\cdots$ &   $\cdots$  \\
\hline
$mk$ &   $ (1)(m+1) + (k-1)(m+2)$ \\
\hline
$mk+1$ &   $ (k)(m+2)$ \\
\hline
\end{tabular}
\caption{Maximum Number of Factors of length $l$ in $f_{\infty;a,b}$}
\label{factorsanym}
\end{table}

Note that some of the factors created by a factor of length $m$ of $f_{\infty;u,v}$ may repeat in the factors created by a factor of length $m+1$ of $f_{\infty;u,v}$. Hence, the total number of factors obtained (i.e. the last column of Tab. \ref{factorsanym}) is, in fact a bound.

Though the proof uses an iterative argument over $m$, in practical situations, when we require the number of factors of a given length $l$, we should fix $m \ge 2$ as the least integer such that $km \ge l$. This is clear from the fact that, a factor of length $l$ of $f_{\infty;a,b}$ will be created by a factor $w$ of $f_{\infty;u,v}$ only when $|w|k \ge l$. 
\end{proof}

\begin{Remark}
    
 Also while obtaining a general formula for $p_{f_{\infty;a,b}}(l)$, the bounds for the cases $l=2,3,\ldots, k-1$ might have been scaled up. But this can be resolved by a simple manipulation. 
\end{Remark}

\begin{Remark}    
The value of the maximum number of factors (in fact the total number of factors) given by Theorem \ref{T3}, in the degenerate case ($k=1$ and $u=a, v= b$) is $m+1$, the factor complexity of the $1D$ infinite Fibonacci word.
\end{Remark}

We note that, achieving the bound given in Theorem \ref{T3} depends on the selection of $u$ and $v$.

\begin{Example}\label{exam6}
Let $u = abaa, v = aaba$ with $k=4$. Then, $f_{\infty;a,b} = aaba|abaa|aaba|aaba|abaa\cdots$. Markers are used for better readability. Let us find the maximum number of factors of length $10$ in $f_{\infty;a,b}$. As $3.k > l$, $m$ is $3$. Thus,  $p_{f_{\infty;a,b}}(10)  \le (k-1)(m+1) + (1)(m+2) = 3.4+ 1.5 = 17$.

Elaborating further, we have the factors of length $3$ of $f_{\infty;u,v}$ as $vuv, uvv, vvu, uvu$. 

Thus, factors of length $10$ of $f_{\infty;a,b}$,

created by $vuv$: $aabaabaaaa, abaabaaaab, baabaaaaba$

created by $uvv$: $abaaaabaaa, baaaabaaab, aaaabaaaba$

created by $vvu$: $aabaaabaab, abaaabaaba, baaabaabaa$

created by $uvu$: $abaaaabaab, baaaabaaba, aaaabaabaa$\\

The factors of length $4$ of $f_{\infty;u,v}$ as $vuvu, vuvv, uvvu, vvuv,uvuv$. 

Thus, factors of length $10$ of $f_{\infty;a,b}$,

created by $vuvu$: $aabaaaabaa$

created by $vuvv$: $aabaaaabaa$

created by $uvvu$: $aaabaaabaa$

created by $vvuv$: $aaabaabaaa$

created by $uvuv$: $aaabaabaaa$

Observe that, as both $u$ and $v$ start and end with the same symbol $a$, some of the factors are repeated. This happens when a factor is created again by a different arrangement of $u$'s and $v$'s. Such a situation happens in this example and hence $p_{f_{\infty;a,b}}(10)= 15$ only. 
\end{Example} 

\begin{Example}\label{exam7}
 Let $u= abbab$ and $v=baaba = \bar{u}$, the complement of $v$. Let us evaluate $p_{f_{\infty;a,b}}(8)$. As $k=5$ and $l=8$, $m$ is $2$. Hence, $p_{f_{\infty;a,b}}(8) \le 17$.  It is easy to check that all the $17$ factors are distinct and the bound is tight for this selection of $u$ and $v$.    
\end{Example}   

\subsection{Location of the Factors}
After counting and enumerating the factors of $f_{\infty;a,b}$ for a given $l \ge 2$, we can now locate the positions of occurrences of these factors in $f_{\infty;a,b}$.

\begin{Theorem}\label{T4}
Let $k \ge 2$ and $|u|=|v|=k$. Let $m$ denote the length of the factors of $f_{\infty;u,v}$ and let $l$ denote the length of the factors of $f_{\infty;a,b}$. For an $m \ge 2$, let $Loc(fac^i_{m}, f_{\infty;u,v})$,$0 \le i \le m$, be the locations of the $i^{th}$ factor of length $m$ of $f_{\infty;u,v}$  in $f_{\infty;u,v}$. Given an $l \ge 2$, consider the least $m \ge 2$ such that $km \ge l$.  Let $Loc(fac^r_{l}, f_{\infty;a,b})$,$0 \le r < (m+1) (mk-l+1) +(m+2) (l -(m-1)k-1)$, be the locations of the $r^{th}$ factor of length $l$ of $f_{\infty;a,b}$  in $f_{\infty;a,b}$.

Then, for $0 \le j \le m, 0 \le j' \le mk -l$, $r = j(mk-l+1)+j' $
\[Loc(fac^{r}_{l}, f_{\infty;a,b}) =   
      \bigcup_{L^j_m}   \left\{(k.L^j_m - (k-1)) + j'\right\}\]
        
        where the union is taken over all $L^j_m \in Loc(fac^j_{m}, f_{\infty;u,v})$; and

   for $0 \le j \le (m+1), 0 \le j' < l - [(m-1)k +1]$, $r=(m+1)(mk-l+1)+j(l - (m-1)k - 1+j'$,
   \[Loc(fac^{r}_{l}, f_{\infty;a,b}) =   \bigcup_{L^{j}_{m+1}}   \left\{k.L^{j}_{m+1} - j'\right\}\]

        where the union is taken over all $L^j_{m+1} \in Loc(fac^j_{m+1}, f_{\infty;u,v})$.

\end{Theorem}

\begin{proof}
From the counting process we used in the proof of Theorem \ref{T3}, it is easy to observe that, for a given $l$, $(mk-l+1)$ factors of length $l$  are created (in $f_{\infty;a,b}$), by each factor of length $m$ of $f_{\infty;u,v}$, and $(l -(m-1)k-1)$ factors of length $l$  are created (in $f_{\infty;a,b}$) by each factor of length $(m+1)$ of $f_{\infty;u,v}$. This explains the range of the index, `$r$' in $Loc(fac^r_{l}, f_{\infty;a,b})$.

$f_{\infty;u,v}$  being the infinite Fibonacci word over $\{u,v\}$, we know the locations of its factors of length $m \ge 1$. Refer (\ref{1stocc}) for the same. Here, as multiplication operations of the locations are involved, after finding the locations of a factor of $f_{\infty;u,v}$, we shift the values by $1$ before using them. Now, as $u=u_1u_2\cdots u_k$ and $v=v_1v_2\cdots v_k$, if a factor of length $m$ of $f_{\infty;u,v}$ (say, $fac_{m}$) is located at position $L_m \ge 1$ in $f_{\infty;u,v}$, then the factors of length $l$ of  $f_{\infty;a,b}$ will occur at positions $k.L_m - (k-1), k.L_m - (k-2), \ldots k.L_m - (k -(mk -l+1))$. And whenever $fac_{m}$ occurs in $f_{\infty;u,v}$, the same set of factors of length $l$ of $f_{\infty;a,b}$ will occur in $f_{\infty;a,b}$. Hence, for specific values of $j,j'$  such that  $0 \le j \le m, 0 \le j' \le mk -l$,   \[\bigcup_{L^j_m}   \left\{(k.L^j_m - (k-1)) + j'\right\}\]
gives the locations of the factor $fac_{l}^{j(mk-l+1)+j'}$. 

Recall that, factors of length $l$ are formed through factors of length $m+1$ (say $fac_{m+1}$) of $f_{\infty;u,v}$ also. As the starting positions of these factors are $k.L_{m+1} - j',  0 \le j' < l - [(m-1)k +1]$, by a similar argument as above, the second part of the result follows.   
\end{proof}

\begin{Remark}
    As remarked earlier all the factors of length $l$ obtained from $fac_{m}$ and $fac_{m+1}$ need not be distinct. In such a scenario, when an already obtained factor is obtained again through different $j,j'$ values, the location sets of the factor can be combined together.
\end{Remark}

\begin{Example}
    Let us use the set up of Example \ref{exam6} and find the locations of the factor  $aabaaaabaa$ of length $10$ of $f_{\infty;a,b}$. This factor is created by the $4$-length factor $vuvu$ of $f_{\infty;u,v}$. By the indexing process we use, this factor is named as $fac_{10}^{12}$ in $f_{\infty;a,b}$. 
    Now from Example \ref{exam5} the locations set of the factor $vuvu$ in $f_{\infty;u,v}$ is, $Loc(fac^4_{4}, f_{\infty;u,v}) = \{4,12,17,25,33,\ldots\}$. Now, using $\bigcup\limits_{L^{j}_{m+1}}   \left\{(k.L^{j}_{m+1} - j'\right\}$ with appropriate values, we have , $Loc(fac^{12}_{10}, f_{\infty;a,b}) = \{16,18,68,100,132,\ldots\}$.
    
\end{Example}

\section{Fibonacci Sequence of \texorpdfstring{$2D$}{} Words}\label{sec_2d_uv_words}
Similar to the Fibonacci sequence of $1D$ words, one can construct a Fibonacci sequence of $2D$ words. We will outline the process here. 

In the development of the Fibonacci sequence of $1D$ words, one might have observed that the sequence can be obtained in two ways. One can first develop the sequence of Fibonacci words $v,vu,vuv,vuvvu,\ldots$ over the alphabet $\{u,v\}$ and thereafter replace $u$ and $v$, respectively by $u_1u_2\cdots u_k$ and $v_1v_2\cdots v_k$. In the second way of construction, we start with the words $u_1u_2\cdots u_k$ and $v_1v_2\cdots v_k$ themselves and concatenate them iteratively in the Fibonacci way to get the sequence of words $v_1v_2\cdots v_k$, $v_1v_2\cdots v_k u_1u_2\cdots u_k$, $v_1v_2\cdots v_k u_1u_2\cdots u_k v_1v_2\cdots v_k$, $\ldots$. 

Similarly a Fibonacci sequence of $2D$ words can be obtained in two ways. First we can develop the sequence of $2D$ Fibonacci words over the alphabet  $\{u,v,w,x\}$, as defined in section \ref{2DFibosection} to get,   

\begin{equation} \label{uvwx}
W_0 = u, \quad W_1 = x, \quad W_2 = \begin{matrix}
x & w\\
v & u
\end{matrix}, \quad W_3 =  \begin{matrix}
x & w & x  \\
v & u & v  \\
x & w & x 
\end{matrix}, \quad W_4 = \begin{matrix}
x & w & x & x & w \\
v & u & v & v & u \\
x & w & x & x & w \\
x & w & x & x & w \\
v & u & v & v & u 
\end{matrix}, \quad \ldots
\end{equation}
and then replace $u,v,w,x$ respectively by $2D$ words over $\{a,b,c,d\}$ of the same size, $(m,n)$. That is, with $u_{i.j},v_{i,j},w_{i,j},x_{i,j} \in \{a,b,c,d\}, 1 \le i \le m, 1 \le j \le n$,
\[u \text{ can be replaced by } \begin{matrix}
u_{1,1} & u_{1,2} & \cdots & u_{1,n} \\
u_{2,1} & u_{2,2} & \cdots & u_{2,n} \\
 & \cdots & \cdots &  \\
u_{m,1} & u_{m,2} & \cdots & u_{m,n} 

\end{matrix}, \quad v \text{ can be replaced by } \begin{matrix}
v_{1,1} & v_{1,2} & \cdots & v_{1,n} \\
v_{2,1} & v_{2,2} & \cdots & v_{2,n} \\
 & \cdots & \cdots &  \\
v_{m,1} & v_{m,2} & \cdots & v_{m,n} 

\end{matrix}, \]

\[w \text{ can be replaced by } \begin{matrix}
w_{1,1} & w_{1,2} & \cdots & w_{1,n} \\
w_{2,1} & w_{2,2} & \cdots & w_{2,n} \\
 & \cdots & \cdots &  \\
w_{m,1} & w_{m,2} & \cdots & w_{m,n} 

\end{matrix}, \quad x \text{ can be replaced by } \begin{matrix}
x_{1,1} & x_{1,2} & \cdots & x_{1,n} \\
x_{2,1} & x_{2,2} & \cdots & x_{2,n} \\
 & \cdots & \cdots &  \\
x_{m,1} & x_{m,2} & \cdots & x_{m,n} 

\end{matrix}. \]

In the other way of construction, initially itself we can take $u,v,w,x$ as $2D$ words of the same size, say, $(m,n)$, and use Definition \ref{DefnApos} with $f_{0,0}= u,f_{0,1}= v,f_{1,0}=w,f_{1,1}=x$ to get the desired sequence of words, $\{W_0,W_1,W_2,W_3,\ldots\}$. 

Note that the sizes of the $2D$ words $u,v,w,x$ all have to be the same for the partial operations $\obar$ and $\ominus$ to be valid. For an easier analysis, similar to what we have assumed in $1D$ setup, we can take $u,v,w,x$ all as square $2D$ words of size $(k,k)$. Then we can easily extend the factor complexity analysis we performed in Section \ref{FacCompFibSeq} to a Fibonacci sequence of $2D$ words.

\begin{Example}
Consider the Fibonacci sequence of $2D$ words as in (\ref{uvwx}). Let $u,v,w,x$ be the $2D$ words as given below. 

$$ u = \begin{matrix}
a & a & b \\
b & b & a \\
b & a & b 
\end{matrix}, \quad v = \begin{matrix}
b & b & a \\
a & a & b \\
a & b & a 
\end{matrix}, \quad w = \begin{matrix}
d & d & c \\
c & c & c \\
d & c & c 
\end{matrix}, \quad x = \begin{matrix}
c & c & d \\
d & d & d \\
c & d & d 
\end{matrix}.$$
Then a few initial words of the Fibonacci sequence of $2D$ words are,
$$
W_0 = \begin{matrix}
a & a & b \\
b & b & a \\
b & a & b 
\end{matrix}, \: W_1 = \begin{matrix}
c & c & d \\
d & d & d \\
c & d & d 
\end{matrix}, \: W_2 =  \begin{matrix}
c & c & d & d & d & c\\
d & d & d & c & c & c\\
c & d & d & d & c & c \\
b & b & a & a & a & b \\
a & a & b & b & b & a \\
a & b & a & b & a & b 
\end{matrix}, \: W_3 =  \begin{matrix}
c & c & d & d & d & c & c & c & d \\
d & d & d & c & c & c & d & d & d \\
c & d & d & d & c & c & c & d & d \\
b & b & a & a & a & b & b & b & a \\
a & a & b & b & b & a & a & a & b \\
a & b & a & b & a & b & a & b & a \\
c & c & d & d & d & c & c & c & d \\
d & d & d & c & c & c & d & d & d \\
c & d & d & d & c & c & c & d & d 
\end{matrix}.$$
\end{Example}

\subsection{Factor Complexity of the Fixed Point}
The fixed point of the above discussed sequence, $W_{\infty,\infty;a,b,c,d}$, can be obtained either directly or from the fixed point, $W_{\infty,\infty;u,v,w,x}$, of the sequence (\ref{uvwx}). 

\begin{Theorem}
    Let $k \ge 2$ and let the sizes of $u,v,w,x$ be $(k,k)$. Let $(m,m')$ denote the size of the factors of $W_{\infty,\infty;u,v,w,x}$ and let $(l,l')$ denote the size of the factors of $W_{\infty,\infty;a,b,c,d}$. Given $l,l' \ge 2$, consider the least $m \ge 2$ such that $km \ge l$ and the least $m' \ge 2$ such that $km' \ge l'$. Then for $l = (m-1)k + i + 1$, $1 \le i \le k$, $m \ge 2$,  and for $l' = (m'-1)k + i' + 1$, $1 \le i' \le k$, $m,m' \ge 2$,  we have 
    \[p_{f_{\infty,\infty;a,b,c,d}}((l,l'))  \le \left[(k-i)(m+1) + (i)(m+2)\right] \left[(k-i')(m'+1) + (i')(m'+2)\right].\]
\end{Theorem}

\begin{proof}
    Before finding the bound for $p_{f_{\infty,\infty;a,b,c,d}}((l,l'))$, $l,l' \ge 2$, observe that every row of $W_{\infty,\infty;a,b,c,d}$ is $f_{\infty;s_1,s_2}$ where $s_1, s_2 \in \{a,b,c,d\}, s_1 \ne s_2$ and there are only $2k$ distinct rows. Similarly every column of $W_{\infty,\infty;a,b,c,d}$, written as a $1D$ Fibonacci word  is $f_{\infty;s_1,s_2}$ where $s_1, s_2 \in \{a,b,c,d\}, s_1 \ne s_2$ and there are only $2k$ distinct columns. This follows from the properties listed in Lemma \ref{lemprop} and the fact that each of $u,v,w,x$ are of size $(k,k)$.

    Given $l,l' \ge 2$, we can find $m \ge 2,m' \ge 
 2$ as the least values such that $l \ge km$, $l' \ge km'$. As the columns of $W_{\infty,\infty;a,b,c,d}$ are $f_{\infty;s_1,s_2}$, in any arbitrary column,  there will be a maximum of $(k-i^0)(m+1) + (i^0)(m+2)$ factors of length $l = (m-1)k+i^0+1$, where $i^0 \in \{1, 2,\ldots, k\}$ (and corresponds to the given $l$). Let us denote this set of factors by `$VF$' and call them `vertical factors'. As the rows of $W_{\infty,\infty;a,b,c,d}$ are $f_{\infty;s_1,s_2}$, in any arbitrary row, there will be a maximum of $(k-i^*)(m'+1) + (i^*)(m'+2)$ factors of length $l' = (m'-1)k+i^*+1$, where $i^* \in \{1, 2,\ldots, k\}$ (and corresponds to the given $l'$). Let us denote this set of factors by `$HF$' and call them `horizontal factors'.  As there are $2k$ distinct columns and $2k$ distinct rows, there will be at the maximum $(2k) ((k-i^0)(m+1) + (i^0)(m+2))$ vertical factors of length $l$ and  $(2k) ((k-i^*)(m+1) + (i^*)(m+2))$ horizontal factors of length $l'$ available in $W_{\infty,\infty;a,b,c,d}$.
 
 Now, let us call the prefix of size $(1,1)$ of a factor (vertical or horizontal) as its head. For any random vertical factor of length $l$, say $VF_I$, available in the $J^{th}$ column of $W_{\infty,\infty;a,b,c,d}$, with its head being positioned in the $I^{th}$ row of $W_{\infty,\infty;a,b,c,d}$, there will be a unique horizontal factor of length $l'$, say $HF_J$, available in the $I^{th}$ row, having its head positioned at the $J^{th}$ column. This argument is similar to the argument used in the proof of Preposition $5$ of \cite{Sivasankar:2023}. Now, $VF_I$ and $HF_J$, having the same head,  will form a $FRAME_{TL}$ with the symbol available at the head becoming \textit{$s_{joint,TL}$}. Similar to the construction used in Lemma \ref{frmframtl}, this $FRAME_{TL}$ can be completed to get a factor of size $(l,l')$ of $W_{\infty,\infty;a,b,c,d}$.  As every vertical factor in $VF$ pairs up with a unique horizontal factor in $HF$, there will be $\left[(k-i^0)(m+1) + (i^0)(m+2)\right] \left[(k-i^*)(m'+1) + (i^*)(m'+2)\right]$ factors of size $(l,l')$. As all these factors being distinct depends on the selection of $u,v,w,x$, we have, 
 \[p_{f_{\infty,\infty;a,b,c,d}}((l,l'))  \le \left[(k-i)(m+1) + (i)(m+2)\right] \left[(k-i')(m'+1) + (i')(m'+2)\right].\]
\end{proof}

\begin{Example}
    Let us extend Example \ref{exam7} here. Let $u,v,w,x$ be any random $2D$ words over $\{a,b,c,d\}$ of size $(5,5)$. Then $p_{f_{\infty,\infty;a,b,c,d}}((8,8)) \le 17.17 = 289$.
\end{Example}

\subsection{Location of the Factors}
The procedure we followed to locate a factor of $f_{\infty;a,b}$ can be extended to locate any factor of $f_{\infty,\infty;a,b,c,d}$. We conclude this section by outlining the steps to perform the same. 

Given a factor `$Fact$' of  $f_{\infty, \infty; a,b,c,d}$ of size $(l,l')$, $l,l' \ge 2$, consider the least $m \ge 2$ such that $km \ge l$ and the least $m' \ge 2$ such that $km' \ge l'$. Let us index all the factors of size $(l,l')$ as $fac_{(l,l')}^{r,r'}$, where $0 \le r < (m+1) (mk-l+1) +(m+2) (l -(m-1)k-1)$ and $0 \le r' < (m'+1) (m'k-l'+1) +(m'+2) (l' -(m'-1)k-1)$. Then using Theorem \ref{T4}, we can find the locations of $FRAME_L$ (i.e. $Loc(fac^{r}_{l}, f_{\infty;s_1,s_2}$,  $s_1 \ne s_2$) in the columns in which $FRAME_L$ occurs as a factor. Similarly, we can find the locations of $FRAME_T$ (i.e. $Loc(fac^{r'}_{l'}, f_{\infty;s_1,s_2}, s_1 \ne s_2$) in the rows in which $FRAME_T$ occurs as a factor. As \textit{$s_{joint,TL}$} occurs at the locations "$Loc(fac^{r}_{l}, f_{\infty;s_1,s_2}) \times Loc(fac^{r'}_{l'}, f_{\infty;s_1,s_2})$", we have,  

\[Loc(fac_{(l,l')}^{r,r'},f_{\infty,\infty;a,b,c,d}) = (Loc(fac^{r}_{l}, f_{\infty;s_1,s_2}),Loc(fac^{r}_{l}, f_{\infty;s_1,s_2})).\]

\section{Concluding Remarks}\label{secconcl}
The knowledge of all the subwords of an infinite word would be very useful to analyse the characteristics of the word. Though any sort of analysis like periodicity, factor complexity is tricky in $2D$ words, $2D$ Fibonacci words with their simple and elegant structure are pliable for exploring their properties. In this paper we have enumerated the subwords of the $2D$ infinite Fibonacci word, $f_{\infty,\infty}$, in a few possible ways. The location of the occurrences of these subwords are also found out.

Suffix tree is an important tool used for pattern matching and dictionary searching \cite{Crochbook:2007, Smyth:2003}. Again, there are some limitations while extending this tool for $2D$ words \cite{Giancarlo:1999}. But the relatively simpler structure of $f_{\infty,\infty}$ may help us to develop one for $2D$ words of similar type. Also, variations attempted in the generation of the Fibonacci sequence \cite{Swain:2017} lead to variants of $1D$\:/\:$2D$ Fibonacci words \cite{Jishe:2011}. We might start exploring these directions. One more compelling direction of work can be towards estimating the factor complexities of  $f_{\infty;a,b}$ ($W_{\infty,\infty;a,b,c,d}$, respectively,) when the length of $u$ and $v$ are not equal in $f_{\infty;u,v}$ (when the sizes of $u$,$v$,$w$,$x$ are not equal in $W_{\infty,\infty;u,v,w,x}$, respectively).

\bibliographystyle{splncs03}
\bibliography{refref}

\end{document}